\documentclass{amsart}

\usepackage{amsmath, amsthm, amssymb}
\usepackage{ytableau}

\newtheorem{thm}{Theorem}[section]
\newtheorem{proposition}[thm]{Proposition}
\newtheorem{cor}[thm]{Corollary}

\newtheorem{conj}[thm]{Conjecture}
\newtheorem{rmk}[thm]{Remark}
\newtheorem{lemma}[thm]{Lemma}
\newtheorem{example}[thm]{Example}
\newtheorem{definition}[thm]{Definition}
\newcommand{\M}{\widehat{\mathcal{M}}}
\newcommand{\bfr}{{\bf r}}
\newcommand{\B}{\mathcal{B}}
\newcommand{\spm}{\mathfrak{sp}_{2m}}
\newcommand{\e}{\epsilon}
\newcommand{\vare}{\varepsilon}
\newcommand{\tf}{\tilde f}
\newcommand{\te}{\tilde e}
\newcommand{\ri}{{\bf r}^{-1}}

\def\dnode#1#2{\overset{#1}{\underset{#2}{\circ}}}

\DeclareMathOperator{\outside}{outside}
\DeclareMathOperator{\inside}{inside}
\DeclareMathOperator{\size}{size}
\DeclareMathOperator{\wt}{wt}
\DeclareMathOperator{\cwt}{cwt}
\DeclareMathOperator{\SSOT}{SSOT}
\DeclareMathOperator{\shape}{shape}

\title[Crystal structure on King tableaux and SSOT]{Crystal structure on King tableaux and semistandard oscillating tableaux}
\author{Seung Jin Lee}
\date{}
\begin{document}

\begin{abstract}
In 1976, King \cite{Kin} defined certain tableaux model, called King tableaux in this paper, counting weight multiplicities of irreducible representation of the symplectic group $Sp(2m)$ for a given dominant weight. Since Kashiwara defined crystals \cite{Kas1}, it is an open problem to provide a crystal structure on King tableaux. In this paper, we present crystal structures on King tableaux and semistandard oscillating tableaux. The semistandard oscillating tableaux naturally appears as $Q$-tableaux in the symplectic version of RSK algorithms. As an application, we discuss Littlewood-Richardson coefficients for $Sp(2m)$ in terms of semistandard oscillating tableaux.  
\end{abstract}

\maketitle
\section{Introduction}

King tableau is one of the models for irreducible representation of the symplectic group $Sp(2m)$, defined by King \cite{Kin} in 1976. Compared to Kashiwara-Nakashima tableaux which is another model for the irreducible representation of $Sp(2m)$, King tableaux have advantages that definition is much simpler than Kashiwara-Nakashima tableaux hence easier to remember and count weight multiplicities. However, since the discovery of crystal theory by Kashiwara \cite{Kas1} and Kashiwara-Nakashima tableaux \cite{KN} (which naturally has a crystal structure), it is an open problem to provide a crystal structure on King tableaux. \\

In this paper, we provide such a $\spm$-crystal structure. We first present a bijection between King tableaux and another set of tableaux called semistandard oscillating tableaux which appears as $Q$-tableaux in the symplectic version of RSK algorithms (See Section \ref{app}). Semistandard oscillating tableaux are also closely related to the multiplication rule between irreducible representation of arbitrary weight and $\wedge^\ell \mathbb{C}^{2m}$ (Theorem \ref{dualpieri1}). The above-mentioned bijection is straight-forward, hence we describe the $\spm$-crystal structure on semistandard oscillating tableaux of fixed shape. \\

The main idea for imposing the crystal structure is by presenting a bijection between semistandard oscillating tableaux of empty shape, which corresponds to the irreducible representation of maximal rectangular weight, with certain symmetric matrices known to have Kirillov-Reshetikhin crystal structure by Kwon \cite{Kwo}. Then the $\spm$-crystal structure on semistandard oscillating tableaux of arbitrary shape naturally follows by proving that the bijection and the crystal operators defined on the symmetric matrices behave well enough that crystal operators $\te_i$ and $\tf_i$ on semistandard oscillating tableaux are \emph{local}. In terms of King tableaux, it means that $\te_i$ and $\tf_i$ only depend on and change positions of $i, \overline{i},i+1,\overline{i+1}$. It will be clear that the $\spm$-crystal structure on semistandard oscillating tableaux is compatible with type $A$ crystal structure on semistandard Young tableaux. \\

There are a few applications of the main result. First of all, semistandard oscillating tableaux naturally appears as a $Q$-tableaux of symplectic version of (dual) RSK algorithm, so the crystal structure on these tableaux also provides the double crystal structure on the $\spm \times \mathfrak{sp}_{2g}$-module 
$$\wedge((\mathbb{C}^{2m})^g) \cong \sum_\lambda \pi^{(2m)}_\lambda \otimes \pi^{(2g)}_{\hat \lambda'}$$
where $\lambda$ runs over all partitions which fit inside an $m\times g$ rectangle, $\pi^{(2m)}_\lambda$ is the irreducible representation of $\spm$ indexed by $\lambda$, and $\hat{\lambda}$ is a rectangular complement of $\lambda$ in the $m \times g$ rectangle. The above isomorphism is known by Howe \cite{How}.\\

Another interesting application of semistandard oscillating tableaux is related to the tensor product multiplicity of $V(\lambda)\otimes L(\mu)\downarrow^{GL(2m)}_{Sp(2m)}$ where $V(\lambda)$ is the irreducible representation of $Sp(2m)$ indexed by a partition $\lambda$ and $L(\mu)$ is the irreducible representation of $GL(2m)$ indexed by $\mu$. We conjecture that the number of $V(\nu)$ in the tensor product is the same as the number of semistandard oscillating tableaux $T$ of weight $\mu$ that begin at $\lambda$, end at $\nu$ satisfying $\te_i(T)=0$ for $i>0$ (See Conjecture \ref{chis} for details). We prove this conjecture when $\mu_1\leq 3$, or $\mu$ is a single row. Note that the description of the conjecture is really similar to those of Littlewood-Richardson coefficients, the number of $L(\nu)$ in $L(\lambda)\otimes L(\mu)$. It would be really interesting if one can calculate the number of $V(\nu)$ in $V(\lambda)\otimes V(\mu)$, the Littlewood-Richardson coefficients for $Sp(2m)$, by simply counting a certain subset of semistandard oscillating tableaux described in Conjecture \ref{chis}.\\

There are more applications. For example, since now there is a crystal structure on King tableaux, there is a natural bijection between King tableaux and Kashiwara-Nakashima tableaux of given shape. Note that Sheats \cite{She} constructed a bijection between those tableaux, and the author does not know whether these bijections are the same or not. Also King tableaux recently appeared in the study of generalized exponents, or special cases of Lusztig's $q$-analogue of weight multiplicities \cite{Lus} in \cite{LL} hence it would be interesting if combinatorics of semistandard oscillating tableaux in this paper is related to their work.  \\

The paper is structured as follows: In Section 2, King tableaux and semistandard oscillating tableaux are defined and the main theorems are explicitly stated. In Section 3, we provide background on crystals and in Section 3.4 we describe $\spm$-crystal structure on the set of semistandard oscillating tableaux. For readers who are interested in the description of crystal structure on King's tableaux can only read Section 2, Section 3.4 and the proof of Theorem 4.1. In Section 4, the bijection between King's tableaux and semistandard oscillating tableaux and the bijection between semistandard oscillating tableaux of empty shape and certain set of symmetric matrices are given. In Section 5, we translate crystal structure on the set of symmetric matrices to the semistandard oscillating tableaux of empty shape, and prove that crystal operator $\te_i$ only depends on and changes $i$-th and $(i+1)$-th components of semistandard oscillating tableaux which proves that semistandard oscillating tableaux of any given shape has a crystal structure. In Section 6, we discuss various above-mentioned applications more in depth.

\section{Definitions and the main result}
\subsection{King tableaux}
\begin{definition} For an integer $m$, a King tableau $T$ of shape $\lambda$, $\ell(\lambda)\leq m$, is a filling of Ferrers diagram of $\lambda$ with the letters of the alphabet $1<\overline 1 < 2 < \overline 2 < \cdots < m < \overline m$, such that:

\begin{enumerate}
\item the entries are weakly increasing along rows and strictly increasing down the columns
\item all entries in row $i$ are larger than, or equal to, $i$.
\end{enumerate}
\end{definition}
The condition $(1)$ is often called the column-strict condition, and we shall refer to $(2)$ as the \emph{symplectic condition}. Note that by condition $(2)$, the subtableau $T^{(i)}$ of $T$ consisting of elements less than equal to $\overline i$ has height at most $i$. The \emph{weight} of $T$ is defined by $\alpha=(\alpha_1,\ldots,\alpha_m)$ where $\alpha_i$ is the number of $i$ minus the number of $\overline i$. We also call a weight of $T$ a \emph{crystal weight} of $T$. For a partition $\mu$ with $\ell(\mu)\leq m$, let $K(\mu,m)$ be the set of King tableaux $T$ of shape $\mu$. Note that definition of King tableaux depends on $m$.

\subsection{Semistandard oscillating tableaux}
For a skew partition $\lambda/\mu$ of size $1$ placed at $(i,j)$, let $r(\lambda/\mu)$ be the \emph{row number} $i$. We also denote $r(\mu/\lambda)$ by $-i$ or simply $\overline i$, where we allow the notation $\mu/\lambda$ when $|\mu|+1=|\lambda|$.
\begin{definition}
An oscillating horizontal strip of length $\ell$ is a sequence $S$ of partitions $S=\left (\lambda_0,\ldots,\lambda_\ell \right)$ with weakly decreasing row numbers $r(\lambda_{p}/\lambda_{p-1})$. Equivalently, there exist $1 \leq j\leq \ell$ satisfying
\begin{enumerate}
\item for $p<j$, $\lambda_p\subset \lambda_{p+1}$ with $|\lambda_p|+1=|\lambda_{p+1}|$ and $r(\lambda_{p+1}/\lambda_p)$ is decreasing for $p<j$,
\item for $p\geq j$, $\lambda_{p+1}\subset \lambda_{p}$ with $|\lambda_{p+1}|+1=|\lambda_p| $ and $r(\lambda_{p}/\lambda_{p+1})$ is increasing for $p\geq j$,
\end{enumerate}  
\end{definition}
We write $\size(S)=\ell$, $\inside(S)=\lambda_0$ and $\outside(S)=\lambda_\ell$. We also denote the maximum partition $\lambda_j$ by $S_*$.  

Note that by the condition $(1)$ and $(2)$, skew partitions $\lambda_j/\lambda_0$ and $\lambda_j/\lambda_\ell$ are horizontal strips. Therefore, once $S_0,S_*,S_\ell$ are given, other $S_i$ are determined uniquely so that the row numbers are decreasing. Let $r(S)$ be the sequence $\left( r(\lambda_p/\lambda_{p-1}) \right)_{p=0}^\ell$ where $\ell$ is the size of $S$, called the row sequence of $S$. Then this sequence consists of decreasing integers.\\

For a partition $\lambda$, let $\lambda'$ be the conjugate of $\lambda$. The definition of an oscillating horizontal strip is motivated by the following rule due to Sundaram:
\begin{thm}\cite{Sun}\label{dualpieri1}
Let $\lambda$ be a partition of length $\leq g$ regarded as a dominant weight of $Sp(2g)$. Let $\pi_\lambda$ be the irreducible representation indexed by $\lambda$. Then $\pi_\lambda \otimes \wedge^\ell \mathbb{C}^{2g}$ is the sum of $\pi_\mu$ where the sum runs over all oscillating horizontal strips from $\lambda'$ to $\mu'$ of length $\ell$.
\end{thm}

Now we define the semistandard oscillating tableaux.

\begin{definition}
A skew semistandard oscillating tableaux is a sequence $T$ of oscillating horizontal strips $(S_1,S_2,\ldots)$ such that $\outside(S_i)=\inside(S_{i+1})$ for $i \in \mathbb{Z}_{>0}$ and length of $S_k$ is zero for $k$ large. We denote $\inside(T)=\inside(S_1)$ and $\outside(T)=\outside(S_k)$ for $k$ large. A weight $\wt(T)$ of $T$ is a sequence $\left(\size(S_1),\size(S_2),\ldots \right)$. When $\inside(T)$ is the empty partition, then we call $T$ a semistandard oscillating tableaux, or SSOT in short. In addition, if the weight of $T$ is $(1,1,\ldots,1)$, then $T$ is called a standard oscillating tableaux, or SOT in short. We denote $\SSOT(\mu)$ a set of semistandard oscillating tableaux $T$ with $\outside(T)=\mu$.
\end{definition}
For a skew semistandard oscillating tableau $T$, define $T^{(p)}$ by $S_{p}$ and $T^{(p)}_i$ by $\lambda_i$ where $S_p=(\lambda_0,\ldots,\lambda_\ell)$. We also define $T^{(p)}_*$ by the maximum partition in $S_p$. We define a \emph{row sequence} $\bfr(T)$ by the concatenation of $r(T^{(p)})$.\\

Let $\SSOT(\mu,m)$ be a subset of $\SSOT(\mu)$ consisting of semistandard oscillating tableaux of weight $\alpha$ satisfying $\alpha_i=0$ for $i>m$. For a SSYT $T$, define $c(T)$ by the maximum of the number of columns of partitions appearing in $T$. Let $\SSOT_g(\mu,m)$ be the subset of $\SSOT(\mu,m)$ consisting of SSOT $T$ satisfying $c(T)\leq g$. For $T \in \SSOT_g(\mu,m)$, we define a \emph{crystal weight} $\beta=(\beta_1,\ldots,\beta_m)$ by $\beta_i= g-\alpha_i$ for all $i=1,\ldots,m$ where $\alpha=(\alpha_1.\ldots,\alpha_m)$ is a weight of $T$. The bijections presented in this paper will be crystal-weight preserving.

\begin{example} \label{example1}
Consider the semistandard oscillating tableau $T=\left(T^{(1)},T^{(2)},T^{(3)},T^{(4)}\right)$ where

\begin{align*}
T^{(1)}&= \left(\emptyset, \ydiagram{1},\emptyset \right)\\
T^{(2)}&= \left(\emptyset, \ydiagram{1},\ydiagram{2},\ydiagram{1} \right)\\
T^{(3)}&= \left(\ydiagram{1},\ydiagram{1,1},\ydiagram{2,1},\ydiagram{2} \right)\\
T^{(4)}&= \left(\ydiagram{2},\ydiagram{2,1},\ydiagram{3,1} \right).\\
\end{align*}
 Then $c(T)=3$ so that $T$ is an element in $\SSOT_3((3,1),4)$ and its row sequence $\bfr(T)$ is $(1 \overline 1) ( 1 1 \overline 1) ( 2 1 \overline 2) ( 2 1)$. We also have 
\begin{align*}
T^{(3)}_0&=\ydiagram{1} \\
T^{(3)}_*&=T^{(3)}_2=\ydiagram{2,1}\\
 T^{(3)}_3&=\ydiagram{2}.
\end{align*}
A weight of $T$ is $(2,3,3,2)$ and a crystal weight of $T$ is $(1,0,0,1)$ when $g=3$.
\end{example}

\subsection{RSK algorithm and symmetric matrices} Let $M_{m,n}$ be the set of $m\times n$ matrices with non-negative integer entries. In this subsection, we define column version of RSK algorithm for $M_{m,n}$ and symmetric matrices.\\
 For $M=(m_{ij}) \in M_{m,n}$, we define a two-line array $w_M$ defined by
$$w_M=\left( \begin{matrix} i_1& i_2 & \cdots & i_\ell \\ j_1 & j_2& \cdots & j_\ell \end{matrix}\right)$$
where $(1)$ $i_1\leq i_2 \leq \cdots \leq i_\ell$, $(2)$ if $i_a=i_b$ and $a<b$, then $j_a\geq j_b$, $(3)$ for each pair $(i,j)$ the number of $r$ satisfying $(i_r,j_r)=(i,j)$ is equal to $m_{ij}$.
It is clear that $M$ uniquely determines $w_M$ and conversely, a two-line array satisfying $(1)-(3)$ and $1\leq i_r \leq m$, $1\leq j_r \leq n$ for all $r$ produces a unique matrix $M$.\\

For $M \in M_{m,n}$, let $(P(M), Q(M))$ the column-RSK image of the matrix $M$. Namely, $P(M)$ is defined by $(j_\ell \rightarrow \cdots(j_3 \rightarrow ( j_2 \rightarrow j_1)))$ where $a\rightarrow S$ is the column insertion of a number $a$ into a semistandard tableau $S$. $Q(M)$ is a semistandard Young tableau defined inductively by recording shapes appearing in $(i_\ell \rightarrow \cdots(i_3 \rightarrow ( i_2 \rightarrow i_1)))$. Note that $Q(M)$ can be defined by $P(M^\top)$ where $M^\top$ is a transpose of $M$. $P(M)$ (resp. $Q(M)$) is called the $P$-tableaux (resp. $Q$-tableaux) of $M$. See \cite{Sta} for more details. \\

The column-RSK also satisfies the following Schensted's theorem \cite{Sch}: The length of the first row of $P(M)$ is the same as the length of the longest (weakly) decreasing subsequence of the second row of $w_M$.

\begin{example}\label{example2}
Let $m=n=4$ and consider the following matrix:
$$M=\begin{bmatrix}2&1&0&1 \\ 1&0 &1 &0 \\ 0&1&0&0 \\ 1&0&0&0\end{bmatrix}.$$
Then the corresponding two-line array is
$$w_M= \left(\begin{matrix} 1&1&1&1& 2&2&3&4 \\ 4&2&1&1&3&1&2&1 \end{matrix} \right).$$
Since $M$ is symmetric, $P(M)$ and $Q(M)$ are the same semistandard Young tableau, which is 
$$\begin{ytableau} 1&1&1&1&2&4\\2&3 \end{ytableau}$$
Note that Schensted's theorem holds for this example.
\end{example}

Define
$$\widehat{\mathcal{M}}_n=\{M=(m_{ij}) \in M_{n \times n} \mid m_{ij}=m_{ji} \text{ and } 2|m_{ii} \text{ for } i,j \in [n] \}.$$
Since an element $M \in \widehat{\mathcal{M}}_n$ is symmetric, we have $P(M)=Q(M)$. Moreover, the condition $2|m_{ii}$ for all $i$ is equivalent to the statement that the shape of $P(M)$ only has even rows. Indeed, this follows from the fact that a permutation obtained from \emph{standardization} \cite[page 321]{Sta} of two-line array $w_M$ has no fixed point and \cite[Exercise 7.28. a.]{Sta}. \\

For a matrix $M$ in $\M_n$, define $c(M)$ by the number of columns of $P(M)$. Let $\M_n^{g}$ be the subset of $\M_n$ consisting of matrices $M$ such that $c(M)$ is less than or equal to $2g$. Main motivation for considering $\M_m$ and $\M_m^g$ is that there are crystal structures on those sets. More precisely, Kwon \cite{Kwo} showed that $\M_m^g \otimes T_{g\Lambda_1}$ is isomorphic to Kirillov-Reshetikhin crystal ${\bf B}^{m,g}$ of type $C^{(1)}_m$ where $T_{\lambda}= \{t_\lambda \}$ is a crystal for a weight $\lambda$ with $\wt(\lambda)=\lambda$ and $\vare_i(t_\lambda)=\varphi_i(t_\lambda)=-\infty$ for $i=1,\ldots,m$. Since we only need a $\spm$-crystal structure on $\M_m^g \otimes T_{g\Lambda_1}$, this crystal structure will be described in Section 3.

\subsection{Main Theorems}
For a partition $\mu$ contained in $(m)^g$, let $\hat \mu$ be the complement partition of $\mu$ in $(m)^g$. For example, when $m=g=3$ and $\mu=(3,1)$, we have $\hat \mu=(3,2)$.\\

Now we are ready to state following main theorems:

\begin{thm} \label{main1} Let $\mu$ be a partition contained in $(m)^g$. Then there exist a crystal-weight-preserving bijection $\Psi_{m,g}$ between the following sets:
\begin{enumerate}
\item $K(\mu,m)$
\item $\SSOT_g(\hat \mu ,m)$
\end{enumerate}
\end{thm}
\begin{thm}\label{main2} There is a $\mathfrak{sp}_{2m}$-crystal structure on both $K(\mu,m)$ and $\SSOT_g(\hat \mu,m)$ isomorphic to the crystal $\B(\mu)$ of irreducible representation of $\mathfrak{sp}_{2m}$ indexed by $\mu$.
\end{thm}
We will give a precise crystal structure on the set of $\SSOT_g(\hat\mu, m)$ after introducing the basics of crystals in Section \ref{crystal}. The way to obtain the crystal structure on SSOT is by analyzing the case $\mu =(m)^g$ and make a bijection between $\SSOT_g(\emptyset ,m)$ with the set $\M^g_m$ which has a crystal structure.
\begin{thm} \label{cor1}Let $\mu = (m)^g$. Then there exist crystal-weight-preserving bijections between:
\begin{enumerate}
\item $K(\mu,m)$
\item $\SSOT_g(\emptyset,m)$
\item $\M^g_m\otimes T_{g\Lambda_1}$
\end{enumerate}
Moreover, there is a crystal graph $\B( (m)^g)$ whose vertices are one of the above-mentioned sets.
\end{thm}

A crystal-weight of an element $M$ in $\M^g_m\otimes T_{g\Lambda_1}$ will be described together with the crystal structure on $\M^g_m\otimes T_{g\Lambda_1}$ in Section 3.\\

Theorem \ref{cor1} can be proved once we provide bijections, but to show that $\SSOT_g(\hat \mu ,m)$ for any $\mu \in (m)^g$ also has the crystal structure, we need to show that the Kashiwara operators $\te_i,\tf_i$ acting on $T=(T^{(1)},\ldots,T^{(m)}) \in \SSOT_g(\emptyset,m)$ are local: Those operators only depend on $T^{(i)}$ and $T^{(i+1)}$ (when $i=0$, they depend on $T^{(1)}$), and only change $T^{(i)}$ and $T^{(i+1)}$. We will prove such results in Section 5, and then prove Theorem \ref{main2}.
\section{Crystals and RSK algorithm}\label{crystal}

\subsection{Background} For Dynkin diagram and a weight lattice for type $C_m$, we follow notation in \cite{Lec}.
We choose to label the Dynkin diagram of $\mathfrak{sp}_{2m}$ by 
\begin{align}\label{diagram}
\dnode{}{0}\Rightarrow \dnode{}{1}-\cdots-\dnode{}{m-1}
\end{align}

and let $I$ be $\{0,1,\ldots,m-1\}$. The weight lattice $P_m$ of $\spm$ can be identified with $\mathbb{Z}^m$ with coordinate basis $\e_{\overline i}$ for $i=1,\ldots, m$. We take the following simple roots:
$$\alpha_0= 2 \e_{\overline 1} \text{ and } \alpha_i=\e_{\overline{i+1}} - \e_{\overline i}, i=1,\ldots, m-1.$$

Then the set of positive roots of $\spm$ is
$$\{ \e_{\overline i} - \e_{\overline j},\e_{\overline i} + \e_{\overline j}\mid 1\leq j <i\leq m \} \cup \{ 2\e_{\overline i} \mid 1\leq i \leq m\}$$
We also have simple coroots defined by 
$$\alpha_0^\vee= \e_{\overline 1} \text{ and } \alpha_i^\vee=\e_{\overline{i+1}} - \e_{\overline i}, i=1,\ldots, m-1$$
where we identify coweight lattice by $\mathbb{Z}^m$. Denote the set of dominant weights of $\spm$ by $P^+_m$. We have fundamental weights $\Lambda_i = \e_{\overline m} + \cdots + \e_{\overline  {i+1}}$ for $i \in I$. For $\lambda \in P^+_m$, write $\lambda$ as a linear combination of fundamental weights: $\lambda= \sum_{i=0}^{m-1} \tilde \lambda_i \Lambda_i$ with $\tilde \lambda_i \in \mathbb{N}$. We will identify a dominant weight $\lambda$ with a partition $(\lambda_{\overline {m}},\ldots, \lambda_{\overline {1}})$ where $\lambda= \lambda_{\overline 1}\e_1 + \cdots + \lambda_{\overline m} \e_{\overline m}$. For example, a partition $(2,1,1)$ will be identified with a dominant weight $\Lambda_{m-1}+\Lambda_{m-3}= 2\e_{\overline m}+\e_{\overline {m-1}}+\e_{\overline {m-2}}$.
\subsection{Crystals}
We give a brief review on crystals \cite{BS,Kas1,KH}. A \emph{crystal} of type $C_m$, also called a $\spm$-crystal, is a nonempty set $\mathcal{B}$ together with maps
\begin{align*}
\te_i,\tf_i&: \B \rightarrow \B \cup \{0\},\\
\vare_i,\varphi_i&: \B \rightarrow \mathbb{Z} \cup \{-\infty\},\\
\wt&:\mathcal{B}\rightarrow P_m \\
\end{align*}
where $i\in I$ satisfying the following conditions:
\begin{enumerate}
\item If $x,y\in \B$ then $\te_i(x)=y$ if and only if $\tf_i(y)=x$. In this case, we assume that 
$$\wt(y)=\wt(x)+\alpha_i, \quad \vare_i(y)=\vare_i(x)-1,\quad \varphi_i(y)=\varphi_i(x)+1. $$
\item We require that
$$\varphi_i(x)= \langle \wt(x),\alpha_i^\vee \rangle + \vare_i(x)$$
for all $x\in \B$ and $i\in I$. If $\varphi_i(x)=\vare_i(x)=- \infty$, we assume that $\te_i(x)=\tf_i(x)=0$.
\end{enumerate}

The operators $\te_i$ and $\tf_i$ are called crystal operators. An element $v\in B$ satisfying $\te_i(v)=0$ for any $i\in I$ is called a highest weight vector and $\wt(v)$ is called the highest weight. A $\spm$-crystal $\B$ is said to be semiregular if for any $x\in \B$ and $i\in I$, we have
$$\vare_i(x) =\text{max} \{ k \mid \te_i^k(x)\neq 0 \}, \quad \varphi(x)=\{ k \mid \tf_i^k(x)\neq 0 \}$$
Then the weight of the vertex $u$ is defined by $\text{wt}(x)=\sum_{i=0}^{m-1} ( \varphi_i(x)- \vare_i(x) ) \Lambda_i$. All $\spm$-crystals appearing in this paper are semiregular. One can associate a $I$-colored directed graph, called a \emph{crystal graph}, for a crystal $\B$ in a following way: the graph has a vertex set $\B$ and there is an edge $x\overset{i}\rightarrow y$ if and only if $\tf_i(x)=y$. \\

The theory of crystal is originated from a crystal basis for modules of a quantum group. For a dominant weight $\lambda$, let $\B(\lambda)$ be the crystal with the unique highest weight $\lambda$ corresponding to the irreducible highest weight $U_q(\spm)$-module with the highest weight $\lambda$. See \cite{KH} for details. For example, a crystal $\B(\Lambda_m)$ has the following crystal graph:
$$\overline{m} \overset{m-1}\rightarrow \overline{m-1} \overset{m-2}\rightarrow \cdots \overset{1}\rightarrow \overline{1} \overset{0}\rightarrow 1 \overset{1}\rightarrow \cdots \overset{m-1}\rightarrow m$$
where $\wt(\overline{i})= \e_{\overline i }$ and $\wt(i)=-\e_{\overline i }$.\\

The action of $\te_i$ and $\tf_i$ on a tensor product of crystals $\B\otimes \B'$ is given by:
\begin{align*}
\wt(u\otimes v)&= \wt(u)+\wt(v)\\
\tf_i(u\otimes v ) &= \begin{cases} \tf_i(u) \otimes v \text{ if } \varphi_i(u)>\vare_i(v)\\ u \otimes \tf_i(v) \text{ if } \varphi_i(u)\leq\vare_i(v) \end{cases}\\
\te_i(u\otimes v ) &= \begin{cases} \te_i(u) \otimes v \text{ if } \varphi_i(u)\geq\vare_i(v)\\ u \otimes \te_i(v) \text{ if } \varphi_i(u)<\vare_i(v) \end{cases}.
\end{align*}
A set $B$ is said to have a $\spm$-crystal structure if there is a $\spm$-crystal with the underlying set $B$.

\subsection{Crystal structure on $\M^g_m$}
Recall that $\M_m^{g}$ is the subset of $\M_m$ consisting of matrices $M$ such that the number of columns in $P(M)$ is less than or equal to $2g$. In this subsection, we give $\spm$-crystal structure on $\M_m^g$ described in \cite{Kwo}.\\

First, we define a weight $\wt(M)$ of $M \in \M_m$ by the following: a coefficient of $\e_{\overline i}$ in $\wt(M)$ is minus the $i$-th row sum of $M$. For example, the weight of $M$ in Example \ref{example2} is $-\left(4\e_{\overline{1}}+2\e_{\overline{2}}+\e_{\overline{3}}+\e_{\overline{4}}\right)$. On the other hand, an element $M \in \M_m^g$ can be identified with an element $M\otimes t_{g\Lambda_1}$ in $\M_m^g \otimes T_{g\Lambda_1}$. Here the weight of $M\otimes t_{g\Lambda_1}$ is $\wt(M)+g\Lambda_1$. We call this weight a \emph{crystal weight} of $M$, and denote by $\cwt(M)$.\\

To define crystal oprators $\te_i,\tf_i$ acting on $\M_m^g \otimes T_{g\Lambda_1}$, we recall the famous crystal operators $e_i,f_i$ for $i\geq 1$ acting on semistandard Young tableaux.
For a semistandard Young tableau $T$ with entries in $[m]$, consider its row reading word ${\bf r}(T)$ by reading entries of $T$ from the first row to the last row, from the right to the left. To define $e_i,f_i$ for $i=1,\ldots,m-1$, we delete all numbers in the row reading word ${\bf r}(T)$ except $i,i+1$. We pair all possible $(i,i+1)$ and delete those pairs to obtain a word $(i+1)\cdots(i+1) i i \cdots i$. It is well-known that as long as we keep pairing adjacent $i$ and $i+1$ among unpaired numbers, the positions of all letters in the unpaired word $(i+1)\cdots(i+1) i i \cdots i$ does not change, hence the order of pairing does not matter. The image $e_i(T)$ (resp. $f_i(T)$) is defined by replacing the last $i+1$ by $i$ (resp. the first $i$ by $i+1$) in this sequence.\\

\begin{rmk} \label{rset} The operator $e_i,f_i$ can be determined if one knows row numbers of $i$ and $i+1$. For example, if 
$$T=\begin{ytableau} 1&1&1&2&2\\2&3 \\3&4\end{ytableau}$$
then the set $rset(T,2)$ of row numbers of $2$ is $\{2,1,1\}$, written in a decreasing order, and the set $rset(T,3)$ of row numbers of $3$ is $\{3,2\}$. Then we know that a word obtained from ${\bf r}(T)$ by deleting all numbers except $2,3$ is $22323$. Note that a pairing $(i,i+1)$ corresponds to a number $p$ in $rset(T,2)$ and a number $q$ in $rset(T,3)$ with $p<q$. We have
$$f_2(T)=\begin{ytableau} 1&1&1&2&3\\2&3 \\3&4\end{ytableau}$$
and $e_2(T)=0$.
\end{rmk}

For $i=1,\ldots, m-1$ and $M \in \M_m^g \otimes T_{g\Lambda_1}$, we define $\te_i(M)$ by the unique matrix $N$ satisfying
$$P(N)=e_i (P(M)).$$
Note that we are also acting $e_i$ on the $Q$-tableau of $M$. We define $\tf_i(M)$ similarly for $i=1,\ldots, m-1$. \\

For $i=0$ and $M=(m_{ij}) \in \M_m^g \otimes T_{g\Lambda_1}$, we simply define $\te_0(M)$ by replacing $m_{11}$ by $m_{11}+2$ if the resulted matrix is still in $\M_m^g$, and zero otherwise. Similarly, $\tf_0(M)$ is obtained by subtracting 2 from $m_{11}$ unless $m_{11}=0$, and zero if $m_{11}=0$.\\

For later purposes, we describe $\te_i(M)$ for $i>0$ without mentioning $P$-tableaux of $M$. For a matrix $M$ with non-negative integer entries with the finite support, we describe a matrix $e_i^Q(M)$ such that $P(e_i^Q(M))= P(M)$ and $Q(e_i^Q(M))=Q(\te_i(M))$. In the two-array notation of $M$, first delete all entries $p \choose q$ unless $p$ is either $i$ or $i+1$. Then pair all $i \choose p$ and $i+1 \choose q$ such that $p< q$ (Compare this with Remark \ref{rset}) and delete these pairs. The leftover is
$$\left( \begin{matrix} i & \cdots &i &i+1 &\cdots &i+1 \\ a_1 &\cdots & a_{\ell_1} & b_1 & \cdots &b_{\ell_2} \end{matrix}\right)$$
where the second line is a weakly decreasing sequence, and two-line array of $e_i^Q(M)$ is obtained from $w_M$ by replacing $i+1 \choose b_1$ by $i \choose b_1$. Note that $e_i^Q(M)$ only depends on $i$-th and $i+1$-th rows of $M$.

Similarly, we can define a matrix $e_i^P(M)$ such that $P(e_i^P(M))= P(\te_i(M))$ and $Q(e_i^P(M))=Q(M)$. 
Note that $e_i^P(M)$ can be defined by the transpose of $e_i^Q(M^\top)$. It is clear that two operators $e_i^P$ and $e_i^Q$ commute, and 
$$e_i^P\circ e_i^Q(M)=\te_i(M).$$ 
Hence $e_i^P(M)$ only depends on $i$-th and $i+1$-th columns of $M$ and other columns does not change when acting $e_i^P$. Also note that when $M$ is symmetric, computing either $e_i^P(M)$ or $e_i^Q(M)$ is enough to compute $\te_i(M)$. For example, if two-line array of $e_i^Q(M)$ is obtained from $w_M$ by replacing $i+1 \choose b_1$ by $i \choose b_1$, then two-line array of $\te_i(M)$ is obtained from $w_{e^Q_i(M)}$ by replacing $b_1 \choose i+1$ by $b_1\choose i$ and rearranging.

\subsection{Crystal structure on SSOT}
For a partition $\mu$ contained in $(m)^g$, let $T=(T^{(1)},\ldots, T^{(m)}) \in \SSOT_g(\hat \mu,m)$. The purpose of this subsection is to describe $\te_i(T)$ and $\tf_i(T)$ for $i=0,1,\ldots,m-1$. As mentioned before, when $\te_i$ or $\tf_i$ is acting on $T$, the operator only depends on $T^{(i)}$ and $T^{(i+1)}$ and only changes $T^{(i)}$ and $T^{(i+1)}$. Therefore, we only need to determine $T'^{(i)}$ and $T'^{(i+1)}$ where $T'=\te_i(T)$ or $\tf_i(T)$.\\

To describe the crystal operators, we use row reading words of $T^{(i)}$ and $T^{(i+1)}$.\\

For $i=0$, $\tf_0(T)$ is obtained from $T$ by adding $1\overline{1}$ in $\bfr(T^{(1)})$ when possible, and $\te_0(T)$ is obtained from $T$ by deleting one $1$ and one $\overline{1}$ in $\bfr(T^{(1)})$ when possible. When it is impossible, we set $\tf_0(T)$ (or $\te_0(T)$ resp.) to be zero. Recall that $g$ is the maximun number of columns in any partition appearing in $T$, hence in $\bfr(T^{(1)})$ there are at most $g$ number of $1$'s. Also $i$ and $\overline{i}$ for $i>1$ do not appear in $\bfr(T^{(1)})$.\\

For $i>0$, we need a few notations for multisets. For a multiset $A$ of integers and an integer $n$, let $-A$ be the set $\{-x \mid x \in A \}$ and $A+n$ be the set $\{x+n\mid x \in A\}$. For two multisets $A$ and $B$ of integers, define a multiset $A\cap B$ so that the number of $n$ in the multiset $A\cap B$ is the mininum of the number of $n$ in $A,B$. Define a multset $A\uplus B$ by a disjoint union of $A$ and $B$. We also define a multiset $A\backslash B$ so that the number of $i$ in $A\backslash B$ is the maximum between zero and (the number of $i$ in $A$ minus the number of $i$ in $B$). One can show the equality $A=(A\backslash B) \uplus (A\cap B)$.\\

Define the multiset $A\uparrow B$ by $(A\backslash B)\uplus (A\cap B+1)$ and $A\downarrow B$ by $(A\backslash B)\uplus (A\cap B-1)$. The operations $\uparrow$ and $\downarrow$ can be reversible in a following way:

\begin{lemma} \label{updown} For multisets $A,B,C,D$ of integers,
\begin{align*}
C&=A\uparrow B\\
D&=B\uparrow A\\
\end{align*}
if and only if
\begin{align*}
A&=C\downarrow D\\
B&=D\downarrow C\\
\end{align*}
\end{lemma}
\begin{proof} It directly follows from the fact that $A\backslash B$ and $B \backslash A$ are disjoint multisets, hence $(A\uparrow B) \cap (B\uparrow A) = A\cap B+1$ and $(A\uparrow B) \backslash (B\uparrow A) = A\backslash B$.

\end{proof}

Now we describe $T'^{(i)}$ and $T'^{(i+1)}$ for $i>0$. Let $\bfr^+(T^{(i)})$ be the set of positive numbers in $\bfr(T^{(i)})$ and $\bfr^-(T^{(i)})$ be the set of positive numbers in $-\bfr(T^{(i)})$. Therefore, $\bfr(T^{(i)})=\bfr^+(T^{(i)}) \uplus - \bfr^-(T^{(i)})$ where $\uplus$ means a disjoint union.\\
Now we define multisets
\begin{align*}
\overline{\bfr^-}(T^{(i)})&= \bfr^-(T^{(i)}) \uparrow \bfr^+(T^{(i+1)})\\
\overline{\bfr^+}(T^{(i+1)})&= \bfr^+(T^{(i+1)}) \uparrow \bfr^-(T^{(i)})\\
C_i(T)&=\bfr^+(T^{(i)}) \uplus -\overline{\bfr^-}(T^{(i)})\\
D_i(T)&=\overline{\bfr^+}(T^{(i+1)}) \uplus -\bfr^-(T^{(i+1)}) 
 \end{align*}
Note that if $ \bfr^-(T^{(i)})$ and $\bfr^+(T^{(i+1)})$ are disjoint, then $C$ and $D$ are simply $T^{(i)}$ and $T^{(i+1)}$ respectively. Also note that $\overline{\bfr^-}(T^{(i)})$ and $\overline{\bfr^+}(T^{(i+1)})$ determine $\bfr^-(T^{(i)})$ and $\bfr^+(T^{(i+1)})$ by Lemma \ref{updown}.\\

Now we do the pairing procedure on two multiset $(C_i(T),D_i(T))$ as in Remark \ref{rset} when we define $e_i$ on SSYT of given shape: we pair $p$ in $C_i(T)$ and $q$ in $D_i(T)$ when $p<q$ until there is no possible pairing. Only difference between the paring procedure in Remark \ref{rset} and the procedure here is that now $(C_i(T),D_i(T))$ may contain negative numbers. After we found and delete all pairings, every element in the leftover in $C_i(T)$ will be larger than equal to every element in the leftover in $D_i(T)$. Then $(C_i(\te_i(T)),D_i(\te_i(T)))$ (resp. $(C_i(\tf_i(T)),D_i(\tf_i(T)))$) are defined from $(C_i(T),D_i(T))$ by moving the largest unpaired number in $D_i(T)$ to $C_i(T)$ (resp. by moving the smallest unpaired number in $C_i(T)$ to $D_i(T)$). Then $\bfr(T')$ for $T'=\te_i(T)$ or $\tf_i(T)$ is determined.

\begin{example}
Let $m=4,g=3$ and let $T$ be the SSOT in $\SSOT_3((2,1),4)$ in Example \ref{example1}. The row reading word of $T$ is 
$$\bfr(T)=(1 \overline 1) ( 1 1 \overline 1) ( 2 1 \overline 2) ( 2 1).$$
Then we have
$$r(\tf_0(T))=(1{1\overline{1}}\overline{1})(11\overline{1})(21\overline{2})(21).$$
Since $g=3$, we have $\tf_0^3(T) =0$.\\

For $i=2$, we have
\begin{align*}
\bfr^+(T^{(2)})&=\{1,1\}& \bfr^-(T^{(2)})&=\{1\} \\
\bfr^+(T^{(3)})&=\{2,1\}& \bfr^-(T^{(3)})&=\{2\}
\end{align*}
Since $\bfr^-(T^{(2)}) \cap \bfr^+(T^{(3)}) = \{1\}$, we have
\begin{align*}
\overline{\bfr^-}(T^{(2)})&= \{1\} \uparrow \{2,1\}=\{2\} \\
\overline{\bfr^+}(T^{(3)})&= \{2,1\}\uparrow  \{1\}=\{2,2\}\\
C_2(T)&=\{1,1, \overline{2}\}\\
D_2(T)&=\{2,2,\overline{2}\}
\end{align*}
Now we can find two pairs $(1,2)$ and $(1,2)$ in $(C_i(T),D_i(T))$ and the leftover in $C_2(T)$ and $D_2(T)$ is $\{\overline{2}\}$. Then we have
\begin{align*}
C_2(\te_2(T))&=\{1,1,\overline{2},\overline{2}\}\\
D_2(\te_2(T))&=\{2,2\}\\
C_2(\tf_2(T))&=\{1,1\}\\
D_2(\tf_2(T))&=\{2,2,\overline{2},\overline{2}\}
\end{align*}
By reversing the way to get $C_i(T),D_i(T)$ from $\bfr(T)$, we obtain $\bfr(\te_2(T)^{(2)})=11\overline{1}\overline{1}$ and $\bfr(\te_2(T)^{(3)})=11$, and $\bfr(\tf_2(T)^{(2)})=11$ and $\bfr(\tf_2(T)^{(3)})=22\overline{2}\overline{2}$.
\end{example}
\section{Bijections for main theorems}
Recall that $\hat\mu$ is the complement partition of $\mu$ in $(m)^g$. In this section, we construct bijections mentioned in Theorem \ref{main1} and Theorem \ref{cor1}, and prove them except the crystal structure part of Theorem \ref{cor1}.
\begin{thm} [Theorem \ref{main1}] \label{kingssot}
For a partition $\mu \subset (m)^g$, there is a bijection $\Psi_{m,g}$ between $K(\mu,m)$ and $SSOT_g(\hat\mu,m)$.
\end{thm}
\begin{proof} Recall that for $T \in K(\mu,m)$, $T^{(i)}$ denote the subtableau of $T$ consisting of entries less than or equal to $\overline i$. Let $T^{(i,1)}$ denote the subtableau of $T$ consisting of entries less than or equal to $i$. Note that both $T^{(i)}$ and $T^{(i,1)}$ are King tableaux. For given $T \in K(\mu,m)$, we will define a SSOT $\Psi_{m,g}(T)=S$ by characterizing $S^{(i)}_0$ and $S^{(i)}_*$. Until we prove that both $S^{(i)}_0$ and $S^{(i)}_*$ are partitions, we consider them as weak compositions determined by the length of each column. They are inductively defined by the following rule:\\

\begin{enumerate}
\item The horizontal strip $S^{(i)}_*/S^{(i)}$ \emph{does not} contain a box at the $j$-th column if and only if there exists $i$ in $T$ at the $(g-j)$-th column.
\item The horizontal strip $S^{(i)}_*/S^{(i+1)}$ contain a box at the $j$-th column if and only if there exists $\overline i$ in $T$ at the $(g-j)$-th column.
\end{enumerate}
By the construction, we automatically have $\Psi_{i,g}(T^{(i)})= S^{(i)}$. We first show that each $S^{(i)}$ is a partition. It is obvious that length of the $j$-th column of $T^{(i)}$ is the number of $p$ and $\overline{p}$ satisfying $p\leq i$ in the $j$-th column of $T^{(i)}$. On the other hand, the length of the $j$-th column of $S^{(i)}$ is
$$\sum_{p=1}^i \left(1-\chi(T^{(i)},p,g-j)-\chi(T^{( i )},\overline p,g-j) \right)$$
where $\chi(T,x,j)$ is $1$ if $x$ is in the $j$-th column of $T$, and $0$ otherwise. The sum is simply $i$ minus the length of the $(g-j)$-th column of $T^{(i)}$. Therefore, we just proved that $S^{(i)}$ is the rectangular complement of $\shape(T^{(i)})$ in $(i)^g$, so that $S^{(i)}$ is a partition. Indeed, note that we have used the symplectic condition of $T^{(i)}$ so that the $j$-th column length of $S^{(i)}$ are non-negative. The fact that entries of $T^{(i)}$ are weakly increasing along rows is used to confirm that $S^{(i)}$ is a partition for all $i$. We also proved that $\outside(S)=\hat \mu$. Similarly, one can prove that $S^{(i)}_*$ is a partition.

Now it is enough to show that both $S^{(i)}_*/S^{(i)}$ and $S^{(i)}_*/S^{(i+1)}$ are horizontal strips for all $i$ if and only if $T$ has entries strictly increasing along columns. Since there exist at most one $i$ and one $\overline i$ on each $j$-th column of $T$, the proof is straight-forward. 
\end{proof}
Note that Theorem \ref{kingssot} provides a bijection between $(1)$ and $(2)$ in Theorem \ref{cor1}. Therefore, to prove the bijection part of Theorem \ref{cor1} it is sufficient to provide a bijection between $SSOT_g(\emptyset,m)$ and $\M^g_m$.
\subsection{A bijection between $SSOT_g(\emptyset,m)$ and $\M^g_m$.}
Before we present the bijection, it is worth considering the case where the weight of SSOT $T$ is $(1,1,\ldots,1)$. Such a $T$ is called an oscillating tableau in \cite{CDDSY}, and in \cite{CDDSY} the bijection between the set of oscillating tableaux of size $m=2n$ and the set of complete matchings on $[2n]$, or equivalently the set of involutions in the symmetric group $S_{2n}$ with no fixed point, is given. Note that those involutions can be thought as permutation matrices in $\M_m$ with zeros on the diagonal. Moreover, the permutation matrix $M$ corresponding to  $T$ satisfies $c(M)=2c(T)$. Indeed, in \cite[Theorem 6]{CDDSY}, it is shown that $c(T)$ is the same as the crossing number of the corresponding complete matching, and in terms of involution $w$ in $S_{2n}$ it is the same as the half of the maximum number $\ell$ such that $i_1<i_2< \cdots<i_\ell$ and $w(i_1)>w(i_2)>\cdots > w(i_\ell)$. The number $\ell$ is the same as $c(M)$ by Schensted's Theorem. \\

Let $\phi$ be the bijection between the subset of $SSOT_g(\emptyset,m)$ consisting of all $T$ with weight $(1,1,\ldots,1)$ and the set of permutation matrix $M$ with zeros on the diagonal and $c(M)\leq 2g$. We construct a bijection $\Phi$ from $SSOT_g(\emptyset,m)$ to $\M_m^g$ by \emph{semistandardizing} the bijection $\phi$. The result $\Phi(T)$ when $T$ has weight $(1,1,\ldots,1)$ will be the same as $\phi(T)$, hence we will describe $\Phi$ and $\phi$ together.\\

Let $T$ be an element in $SSOT_g(\emptyset,m)$. Let $T'$, the \emph{standardization} of $T$, be the standard oscillating tableau obtained by concatenating components $T^{(i)}$ of $T$. It is clear from the definition that $c(T)=c(T')$. Let $T'=(\lambda_0,\lambda_1,\ldots,\lambda_{m'})$. Then by definition $\lambda_0$ and $\lambda_{m'}$ are the empty partition.\\

 To describe $\phi(T')$, we need to inductively define a pair $(U_i,V_i)$ for $i=0,\ldots, m'$ where $U_i$ is the set of ordered pairs of integers in $[m']$ and $V_i$ is a standard Young tableau of shape $\lambda_i$ Let $U_0$ be the empty set, and let $V_0$ be the empty standard Young tableau.

\begin{enumerate}
\item If $\lambda_i\supset \lambda_{i-1}$, then let $U_i=U_{i-1}$ and $V_i$ is obtained from $V_{i-1}$ by adding the entry $i$ in the square $\lambda_i/\lambda_{i-1}$.
\item If $\lambda_i\subset\lambda_{i-1}$, then $V_i$ is the unique standard Young tableau of shape $\lambda_i$ such that $V_{i-1}$ is obtained from $V_i$ by row-inserting some number $j$. Note that $j$ must be less than $i$. The set $U_i$ is obtained from $U_{i-1}$ by adding the ordered pair $(j,i)$. 
\end{enumerate}
It is clear from the definition that $U_0\subset U_1 \subset \cdots \subset U_{m'}$, and every integer $i$ in $[m']$ appears in $U_{m'}$ exactly once. Therefore $U_{m'}$ represents a complete matching on $[m']$. Then $\phi(T')$ is defined by the permutation matrix corresponding to the complete matching represented by $U_{m'}$. \\

To define $\Phi(T)$ from $U_{m'}$ where weight of $T$ is $\alpha$, we consider the two-line array notation of the involution $w$ in $S_m$ defined by $U_{m'}$. Then we replace $1,2,\ldots,\alpha_1$ by 1, and $\alpha_1+1,\alpha_1+2,\ldots, \alpha_1+\alpha_2$ by $2$, and so on. Then the array obtained from this procedure, which we call \emph{semistandardization}, represents a matrix $M=(m_{i,j})$ in $\M_m^g$, where $m_{i,j}$ is the number of columns of the array equal to $\left( \begin{matrix}i\\j \end{matrix}\right)$. Note that semistandardization of the $P$-tableau of $w$ is the same as a $P$-tableau of $M$, hence we give the same name for the above-mentioned procedure. Also note that a $P$-tableau corresponding to $w$ (resp. $M$) are the same as a $Q$-tableau corresponding to $w$ (resp. $M$) so that semistandardization of $w$ corresponds to semistadardizing both $P$ and $Q$-tableaux of a given weight $\alpha$.
\begin{example}\label{example3}
Consider $T=(T^{(1)},T^{(2)},T^{(3)},T^{(4)})$ given by
\begin{align*}
T^{(1)}&= \left(\emptyset, \ydiagram{1},\ydiagram{2} \right)\\
T^{(2)}&= \left(\ydiagram{2}, \ydiagram{2,1},\ydiagram{2} \right)\\
T^{(3)}&= \left(\ydiagram{2}, \ydiagram{1} \right)\\
T^{(4)}&= \left(\ydiagram{1},\emptyset \right)\\
\end{align*}
where $m=4$ and $g=2$.
Then standardization $T'$ of $T$ is simply
$$T'=\left(\emptyset, \ydiagram{1},\ydiagram{2}, \ydiagram{2,1},\ydiagram{2}, \ydiagram{1},\emptyset \right)$$

Corresponding $(U_i,V_i)$'s are
\begin{center}
\begin{tabular}{l|lllllll}

$ V_i$& $ \emptyset$& $\begin{ytableau} 1\end{ytableau}$   &  $\begin{ytableau} 1&2\end{ytableau}$   &  $\begin{ytableau} 1&2\\ 3\end{ytableau}$   & $\begin{ytableau} 1&3\end{ytableau}$  &  $\begin{ytableau} 1\end{ytableau}$ & $\emptyset$  \\ 
 $U_i$ & &  &  &  &$ (2,4)$&$(3,5)$&$(1,6)$ 
\end{tabular}

\end{center}

Note that $U_m$ corresponds to an involution $w=645231$.\\

Since the two-line array corresponding to $w$ is
$$\left(\begin{matrix} 1&2&3&4&5&6 \\ 6&4&5&2&3&1 \end{matrix} \right),$$
and the weight $\alpha$ of $T$ is $(2,2,1,1)$, the two-line array of $\Phi(T)=M$ is
$$ \left(\begin{matrix} 1&1&2&2&3&4 \\ 4&2&3&1&2&1 \end{matrix} \right)$$
and the matrix $M$ is
$$\begin{bmatrix} 0&1&0&1 \\ 1&0&1&0 \\ 0&1&0&0 \\ 1&0&0&0 \end{bmatrix}$$

One can check that the P-tableau of $M$ and $w$ are
$$P(M)=\begin{ytableau} 1&1&2&4\\ 2&3 \end{ytableau} \quad \textrm{and}\quad P(w)=\begin{ytableau} 1&2&4&6\\ 3&5 \end{ytableau}$$
so that $P(M)$ is a semistandardization of $P(w)$. Also, by considering $M$ as an element in $\M^g_m\otimes T_{g\Lambda_1}$, the crystal weight of $M$ is 
$$(-2\epsilon_{\overline{1}}-2\epsilon_{\overline{2}}-\epsilon_{\overline{3}}-\epsilon_{\overline{4}})+(2\epsilon_{\overline{1}}+2\epsilon_{\overline{2}}+2\epsilon_{\overline{3}}+2\epsilon_{\overline{4}})=\epsilon_{\overline{3}}+\epsilon_{\overline{4}}$$
which is the same as the weight of King tableau $\Psi_{m,g}^{-1}(T)=\begin{ytableau} 2 & \overline{2} \\ 3 & 3 \\ \overline{3} & 4 \\ 4 & \overline{4}\end{ytableau}$ corresponding to $T$.
\end{example}

To show that $\Phi$ is a bijection, it suffices to define an inverse $\Phi^{-1}$. Given $M=(m_{ij}) \in \M_m^g$ with a row sum $\alpha$ satisfying $|\alpha|=m'$, consider a two-line array $w_M$. Then we standardize the array in a usual way to obtain a two-line array of the permutation $w$ in $S_{m'}$: the first line of the new array is $1,2,\ldots, m'$, and for the second line of the new array, we obtained from $w_M$ by replacing $\alpha_1$ number of $1$ by $\alpha_1,\alpha_1-1, \ldots,2,1$ in a decreasing way, and replacing $\alpha_2$ numbers of $2$ by $\alpha_1+\alpha_2,\alpha_1+\alpha_2-1,\ldots,\alpha_1+2,\alpha_1+1$, and so on. Since we have the condition on $w_M$ that $i_a=i_b$ and $a<b$ implies $j_a\geq j_b$ (the second condition that $w_M$ satisfies), corresponding permutation $w$ satisfies $w(1)>w(2)>\cdots>w(\alpha_1)$, $w(\alpha_1+1)>\cdots>w(\alpha_1+\alpha_2)$, and so on. Also by the construction, $w$ is an involution with no fixed point. Indeed, if $w$ has a fixed point, one can make a contradiction by showing that $m_{pp}$ is odd for some $p$. \\

Therefore by taking the inverse $\phi^{-1}$, we obtain a SOT $T'$ of weight $1^{m'}$. Then it suffices to show that we can semistandardize $T'=(\lambda_0,\lambda_1,\ldots,\lambda_{m'})$ to obtain a SSOT $T$ of weight $\alpha$. Namely, we need to show that sequences $(\lambda_0,\ldots,\lambda_{\alpha_1})$, $(\lambda_{\alpha_1+1},\ldots,\lambda_{\alpha_1+\alpha_2})$, $\cdots$ form oscillating horizontal strips. This follows from the property of $\phi^{-1}$ so we now describe how $V_i$ can be reculsively constructed from the permutation $w$. Note that $\lambda_i$ is the shape of $V_i$. \\

We set $V_{m'}$ equal to the empty partition. For $i=0,\ldots,m'-1$, $V_i$ are defined from $w$ as follows:
\begin{enumerate}
\item if $i>w(i)$, then the $V_{i}$ is obtained by row-inserting $w(i)$ into $V_{i+1}$.
\item if $i<w(i)$, then the $V_{i}$ is obtained from $V_{i+1}$ by deleting the entry $i$. Note that $i$ is the biggest entry of $V_{i+1}$. 
\end{enumerate}

Let $\ell$ be the length of $\alpha$ and $\beta_i$ be $\sum_{j=1}^i \alpha_j$. We need to show that for given $i=0,\ldots,\ell-1$, the sequence $(\lambda_{\beta_j +1}, \lambda_{\beta_j +2},\ldots, \lambda_{\beta_{j+1}})$ forms a oscillating horizontal strip. We prove this for $i=0$ for the sake of simplicity since the proof works the same. By the construction of $w$, we have $w(1)>w(2)>\cdots>w(\alpha_1)$. There is a unique $i'\leq \alpha_1$ such that $ i'<w(i')$ but $i'+1>w(i'+1)$. Consider a tableau $V_{\alpha_1}$ and inductively defined $V_i$ for $i<\alpha_1$. Then $V_{i'+1}$ is obtained from $V_{\alpha_1}$ by row-inserting $w(\alpha_1)<w(\alpha_1-1)<\cdots<w(i'+1)$ in a increasing way and by the property of row-insertion, $r(\lambda_j/\lambda_{j+1})$ are negative and decreasing for $j=i'+1,\ldots,\alpha_1$. Also note that $V_0$ is obtained from $V_{i'+1}$ by deleting $i',i'-1, \ldots,2,1$ in a decreasing way, so we need to show that $1,2,\ldots,i'-1$ are not descents of $V_{i'+1}$. Note that the way to define $V_j$ does not change the descent set: if $V_j$ and $V_{j+1}$ contain both $p$ and $p+1$ for some $p$, then $p$ is a descent of $V_j$ if and only if $p$ is a descent of $V_{j+1}$. Since $w(p)>w(p+1)$ for $p=1,\ldots,\alpha_1-1$, the number $p$ is inserted to some $V_j$ before $p+1$ is inserted, so $p$ is not a descent of any $V$ if $p<\alpha_1$. Therefore, we have proved that $(\lambda_0,\cdots,\lambda_{\alpha_1})$ is indeed a oscillating horizontal strip. 

\section{Crystal structure on SSOT}
Let $T$ be a SSOT $(T^{(1)},T^{(2)},\ldots,T^{(m)})$. We define $\te_i(T)$ for $i\in I$ simply by $\Phi^{-1}(\te_i(\Phi(T)))$ and similarly for $\tf_i$. In this section, we prove that the operator $\te_i$ acts locally:

\begin{proposition}\label{loc} For $j\neq i,i+1$,  $T^{(j)}$ does not change when acting $\te_i$ or $\tf_i$ on $T$.  
\end{proposition}

After proving this, we show that $\te_i$ defined here is the same as the one described in Section 3.4.\\

We set up notations first before providing a proof of Proposition \ref{loc}. For a matrix $M=(m_{ij})$ with $\Phi(M)=T$, we define an inverse column word $\ri(M)$ by $j_{m'}j_{m'-1} \ldots j_2j_1$ where
$$w_M=\left( \begin{matrix} i_1& i_2 & \cdots & i_{m'} \\ j_1 & j_2& \cdots & j_{m'} \end{matrix}\right).$$

Also, for a matrix $M$ and an interval $[p,q]\subset [1,m']$, we define a matrix $M^{[p,q]}$ by
$$w_{M^{[p,q]}}=\left( \begin{matrix} i_p& i_{p+1} & \cdots & i_{q} \\ j_p & j_{p+1}& \cdots & j_{q} \end{matrix}\right).$$
For a matrix $M$, we define pairs of semistandard tableaux $(P_q(M),V_q(M))$ in a following way. We set both $P_{m'}(M)$ and $V_{m'}(M)$ to be the empty semistandard Young tableau. For $0\leq q < m'$, $P_q(M)$ is obtained by row-inserting $j_{q+1}$ into $P_{q+1}(M)$. In other words, $P_q(M)=  ((j_{m'} \leftarrow j_{m'-1}) \cdots \leftarrow j_{q+1})$ where $S\leftarrow a$ denote the tableau obtained by row-inserting $a$ into $S$. Note that $P_0(M)$ is equal to $P(M):=(j_{m'}\rightarrow \cdots (j_2\rightarrow j_1))$ because the row-insertion commutes with the column insertion.\\

We define $V_q(M)$ as the semistandardization of $V_i$ defined in Section 4. Explicitely, $V_q(M)$'s are defined inductivley by the following:
\begin{enumerate}
\item If $i_q>j_q$, then $V_q(M)$ is obtained by row-inserting $j_{q+1}$ into $V_{q+1}(M)$,
\item If $i_q<j_q$, then $V_q(M)$ is obtained from $V_{q+1}(M)$ by deleting the right-most $i_q$. Note that $i_q$ is the largest entry of $V_{q+1}(M)$.
\item If $i_q=j_q=r$ for some $r$, then there is an even number $2\ell$ such that $(i_q,j_q)=(i_{q-1},j_{q-1})=\cdots=(i_{q-2\ell+1},j_{q-2\ell+1})=(r,r)$. Then for $1\leq q' \leq \ell$, $V_{q-q'}(M)$ is obtained by row-inserting $r$ into $V_{q-q'+1}(M)$ and for $\ell<q' \leq 2q$, $V_{q-q'}(M)$ is obtained from $V_{q-q'+1}(M)$ by deleting the right-most $r$.
\end{enumerate}

Since $V_i(M)$ is the semistandardization of $V_i$, when $\Phi(M)=T$, we have

\begin{lemma}\label{tiv} $T^{(i)}$ is equal to
$$\left( \shape(V_{\beta_i}),\shape(V_{\beta_i+1}),\ldots, \shape(V_{\beta_{i+1}})\right)$$
where $\beta_i=\sum_{j=1}^{i}\alpha_j$.
\end{lemma}

There is a very useful relation between $P_q(M)$ and $V_q(M)$: 

\begin{lemma}\label{vp} $V_q(M)$ can be obtained from $P_q(M)$ by deleting $m'-q$ number of biggest entries (we delete the right-most ones first). 
\end{lemma}

\begin{proof}
The lemma works because the deleting the biggest entry does commute with row-inserting a smaller number.
\end{proof}

\begin{example}
Continuing Example \ref{example3}, let $M$ be
$$\begin{bmatrix} 0&1&0&1 \\ 1&0&1&0 \\ 0&1&0&0 \\ 1&0&0&0 \end{bmatrix}.$$

Then the inverse column word $\ri(M)$ is $121324$. The pair $(P_q(M),V_q(M))$ are

\begin{center}

\begin{tabular}{c|lllllll}
$ i$ &$0$&$1$&$2$&$3$&$4$&$5$&$6$  \\[1ex]
 $ P_i$&  $\begin{ytableau} 1&1&2& 4\\2&3 \end{ytableau}$& $\begin{ytableau} 1&1&2\\2&3 \end{ytableau}$ &  $\begin{ytableau} 1&1&3\\2\end{ytableau}$   &  $\begin{ytableau} 1&1\\ 2\end{ytableau}$   &$\begin{ytableau} 1&2\end{ytableau}$   &  $\begin{ytableau} 1\end{ytableau}$   &$ \emptyset$  \\[4ex]  
$ V_i$& $ \emptyset$& $\begin{ytableau} 1\end{ytableau}$   &  $\begin{ytableau} 1&1\end{ytableau}$   &  $\begin{ytableau} 1&1\\ 2\end{ytableau}$   & $\begin{ytableau} 1&2\end{ytableau}$  &  $\begin{ytableau} 1\end{ytableau}$ & $\emptyset$  \\ 
\end{tabular}

\end{center}
Note that Lemma \ref{tiv} and Lemma \ref{vp} hold.
\end{example}

Now we are ready to prove Proposition \ref{loc}
\begin{proof}

For a matrix $M=(m_{ij})$ with $\Phi(M)=T$ and $M'=(m'_{ij})=\te_i(M)$ with $\Phi(M')=T'$, consider the two-line arrays corresponding to $M$ and $M'$:
$$w_M=\left( \begin{matrix} i_1& i_2 & \cdots & i_{m'} \\ j_1 & j_2& \cdots & j_{m'} \end{matrix}\right)\quad w_{M'}=\left( \begin{matrix} i'_1& i'_2 & \cdots & i'_{m'} \\ j'_1 & j'_2& \cdots & j'_{m'} \end{matrix}\right)$$
Since $M'= e_i^Q \circ e_i^P(M)$, there is a unique index $r$ such that the two-line array of $e_i^P(M)$ is obtained by replacing one ${i_r \choose i+1}$ by ${i_r \choose i}$, and $w_{M'}$ is obtained from this array by replacing one $i+1 \choose i_r$ by $i \choose i_r$ and rearranging so that the first line of the array is weakly increasing. 

There are three cases to deal with: (1) $i_r<i$ (2) $i_r>i+1$ (3) $i_r=i$ or $i+1$.\\

For $(1)$, note that $P_q(M)$ and $P_q(M')$ are the same for $q\geq \beta_{i+1}$ since first $m-\beta_{i+1}$ letters in inverse column words of $M$ and $M'$ are the same. Therefore corresponding $T^{(j)}$ for $j>i+1$ are the same as $T'^{(j)}$. For $j<i$, first we show that $P_q(M)$ and $P_q(M')$ are still the same if $\beta_{i-1}\geq q>\beta_{i_r}$. Since $M'=\te_i(M)$ and $\te_i(M)=e_i^Q e_i^P(M)$, when $\beta_{i-1}\geq q>\beta_{i_r}$ we have $M'^{[q,m']}=e_i^Q(M^{[q,m']})$ because $e_i^Q$ only depends on $M^{[\beta_{i-1}+1,\beta_{i+1}]}$ and $e_i^P$ changes $i_r$-th row of $M$. Therefore, first $m-q$ numbers in the inverse column words for $M$ and $M'$ are dual Knuth-equivalent for $\beta_{i-1}\geq q>\beta_{i_r}$ and therefore $P_q(M)=P_q(M')$.

 When $q\leq \beta_{i_r}$, $P_q(M)$ and $P_q(M')$ may be different since one $i+1$ is replaced by $i$ from $P_q(M)$ to $P_q(M')$, but the shape of $V_q(M)$ and $V_q(M')$ are the same since all numbers bigger than $i_r$ are deleted from $P_q(M)$ and $P_q(M')$ respectively by Lemma \ref{vp}. Therefore, we proved the proposition for case $(1)$.\\

For (2), when $q\geq \beta_{i+1}$ either $P_q(M)$ and $P_q(M')$ are the same or one $i+1$ is replaced by $i$ from $P_q(M)$ to $P_q(M')$ by $e^P_i$. In both cases, the shape of $V_q(M)$ and $V_q(M')$ are the same. When $q\leq \beta_{i-1}$, $V_q(M)$ and $V_q(M')$ are the same since both $i+1$ and $i$ are deleted so they have the same shape. Case (3) can be proved similarly hence omitted.
\end{proof}
\subsection{Explicit description of crystal operators} In this subsection, we show that $\te_i(T)$ and $\tf_i(T)$ defined by using the bijection $\Phi$ is the same as the ones defined in Section 3.4.\\

For a matrix $M$, let the second row of $w_{M^{[\beta_{i-1}+1,\beta_{i}]}}$ be $x_1x_2\ldots x_p y_1 y_2 \ldots y_q$ where $x_j\geq i$ and $y_j\leq i $ for all $j$ and the number of $i$ in $x_j$'s and $y_j$'s are the same. Let $A_i$ be the ordered set $\{x_1,\ldots,x_p\}$ and $B_i$ be the ordered set $\{y_1,\ldots,y_q\}$. Then we have
$$P_{\beta_{i-1}}(M)=(((P_{\beta_{i+1}}(M) \leftarrow B_{i+1})\leftarrow A_{i+1}) \leftarrow B_i )\leftarrow A_i$$
where $P\leftarrow A_i$ means $(((P \leftarrow a_p) \leftarrow \cdots ) \leftarrow a_1)$, namely we row-insert a smaller number first. For a tableau $P$, let $E_i(P)$ be the tableau obtained by deleting all entries $i$ from $P$ where $i$ is the biggest entry in $P$. Then we have
$$V_{\beta_{i-1}}(M)= E_i(E_{i+1}(V_{\beta_{i+1}}(M) \leftarrow B_{i+1})\leftarrow B_i)$$

Let $r(B_{i})$ be the ordered set of row numbers of horizontal strip obtained from $V_{\beta_i(M)}\leftarrow B_i$, written in a decreasing way. Let $r(A_i)$ be the ordered set of row numbers of $i$ in $V_{\beta_i(M)}\leftarrow B_i$, namely row numbers of horizontal strip obtained from $E_i(V_{\beta_i(M)}\leftarrow B_i)$. Similarly, when $\te_i(M)=M'$, we define $A'_i,B'_i,r(A'_i),r(B'_i)$ in a similar way for $M'$. Then the goal of this subsection is to determine $r(A'_i)$, $r(B'_i)$, $r(A'_{i+1})$, $r(B'_{i+1})$ from $r(A_i),r(B_i),r(A_{i+1}),r(B_{i+1})$ as they are the only information given by the SSOT $T$. Indeed, note that $r(A_i),r(B_i),r(A_{i+1}),r(B_{i+1})$ are the same as $\bfr^+(T^{(i)}),\bfr^-(T^{(i)}),\bfr^+(T^{(i+1)}),\bfr^-(T^{(i+1)})$ defined in Section 3.4.\\

We need two more statistics from $M$. Let $\overline{r}(B_i)$ be the ordered set of row numbers of horizontal strip obtained when inserting $B_i$ to $V_{\beta_{i+1}}(M)\leftarrow B_{i+1}$. Also, let $\overline{r}(A_{i+1})$ be the ordered set of row numbers of $i+1$ in $(V_{\beta_{i+1}}(M)\leftarrow B_{i+1})\leftarrow B_i$. Then we have the following lemma.

\begin{lemma} \label{up} We have
\begin{align*}
\overline{r}(B_i)&= r(B_i) \uparrow r(A_{i+1}) \\ 
\overline{r}(A_{i+1})&=  r(A_{i+1}) \uparrow r(B_i)  
\end{align*} where $A \uparrow B := (A\backslash B) \uplus (A\cap B +1)$.

\end{lemma}
\begin{proof} Let $Q$ be $P_{\beta_{i+1}}(M) \leftarrow B_{i+1}$. Then $r(B_i)$ is the row numbers of $E_{i+1}(Q) \leftarrow B_i$ and $\overline{r}(B_i)$ is the row numbers of $Q\leftarrow B_i$. Note that insertion paths of $ E_{i+1}(Q) \leftarrow B_i$ and $Q\leftarrow B_i$ are exactly same except the case when the last entry in an insertion path of $Q\leftarrow B_i$ is $i+1$. In this case the length of this insertion path in $Q\leftarrow B_i$ is one bigger than the corresponding insertion path in $E_{i+1}(Q) \leftarrow B_i$. However row positions of $i+1$ in $Q$ is given by $r(A_{i+1})$, hence we have a bijection between such insertion paths in $E_{i+1}(Q) \leftarrow B_i$ and the multiset $r(B_i) \cap r(A_{i+1})$, by sending an insertion path to the row number of the last entry of the insertion path. Then the first equation of the lemma follows. \\

For the second equation, note that $\overline{r}(A_{i+1})$ is the row numbers of $i+1$ in $Q\leftarrow B_i$ and $r(A_{i+1})$ is the row numbers of $i+1$ in $Q$. When we insert $B_i$ into $Q$, entries $i+1$ in $Q$ are either staying at the same row or moving down one row. We already showed that $i+1$ in $Q$ is moving down one row if $i+1$ is in the insertion path of $Q\leftarrow B_i$ and in this case row positions of such $i+1$'s is the multiset $r(B_i) \cap r(A_{i+1})+1$. Therefore the second equation follows.

\end{proof}

Recall that the operation $\uparrow$ is reversible because of Lemma \ref{updown}.

\begin{cor} \label{down} We have
\begin{align*}
r(B_i)&=\overline{r}(B_i)\downarrow \overline{r}(A_{i+1})\\
r(A_{i+1})&=\overline{r}(A_{i+1})\downarrow \overline{r}(B_i)
\end{align*}
where $A\downarrow B := (A\backslash B) \uplus (A \cap B -1)$.
\end{cor}

Now we are ready to describe $\te_i(T)$ in terms of $r(A_i),r(B_i),r(A_{i+1}),r(B_{i+1})$. We first need to know whether $i_r$ is less than equal to $i$, or greater than $i$. Recall that to know $i_r$, it is enough to compute $e_i^Q(M)$ by pairing numbers at $i$-th row of $M$, which is $A_i\cup B_i$, and numbers at $i+1$-th rows of $M$ which is $A_{i+1} \cup B_{i+1}$. Let $a_i$, $a_{i+1}$ be the number of unpaired elements in $(A_i,A_{i+1})$ after pairing all $p<q$ pairs in $(A_i,A_{i+1})$, and let $b_i$, $b_{i+1}$ be the number of unpaired elements in $(B_i,B_{i+1})$. Then $i_r \leq i$ if and only if $b_i\geq a_{i+1}$. Note that this condition is essentially the same as the condition for tensor product of crystals in Section 3.2.\\

Now $a_i,a_{i+1},b_i,b_{i+1}$ can be computed from $r(A_i),\overline{r}(B_i),\overline{r}(A_{i+1}),r(B_{i+1})$. The number of all parings $(p,q)$ (with $p<q$) in the set $(A_i,A_{i+1})$ is the same as the number of pairings in the set $(r(A_i),\overline{r}(A_{i+1}))$ since $r(A_i)$ is the ordered set of row positions of $i$ in $(V_{\beta_{i+1}}(M)\leftarrow B_{i+1})\leftarrow B_i$ and $\overline{r}(A_{i+1})$ is the ordered set of row positions of $i+1$ in the same tableau. This can be proved by an induction on $q$ by keeping track of positions of $i$ and $i+1$ in $V_q(M)$. Indeed, note that inserting a number $x$ into $V_q(M)$ with $x\neq i,i+1$, or deleting a bigger number than $i+1$ does not affect the number of parings $(i,i+1)$, so one can assume that $r$-th row of $M$ for $r>i+1$ is zero except $i$-th column and $(i+1)$-th column, then the claim follows since $i$-th and $(i+1)$-th columns are determined by $A_i$ and $A_{i+1}$. Therefore, $a_i$ (resp. $a_{i+1}$) is the cardinality of $A_i$ (resp. $A_{i+1}$) minus the number of such pairings.\\

Similary, the number of pairings $(p,q)$ with $p<q$ in the set $(B_i,B_{i+1})$ is the same as the number of pairings $(p',q')$ with $p'<q'$ in $(r(B_{i+1}),\overline{r}(B_i))$ since $r(B_{i+1})$ is the row numbers of horizontal strip when inserting $B_{i+1}$ into $V_{\beta_{i+1}}(M)$ and $\overline{r}(B_i)$ is the row numbers of horizontal strip when we insert $B_i$ in $V_{\beta_{i+1}}(M)\leftarrow B_{i+1}$. Note that $B_{i+1}$ is inserted before $B_i$, so we have a pairing $(p',q')$ with $p'<q'$ where $p' \in r(B_{i+1})$ and $q' \in \overline{r}(B_i)$. Therefore, the pairing $(p',q')$ becomes a paring $(-q',-p')$ with $-q'<-p'$ with $-q'$ in  $-\overline{r}(B_i)=-\overline{\bfr^-}(T^{(i)})$ and $-p'$ in $-r(B_{i+1})=\bfr^-(T^{(i+1)})$ that appears in Section 3.4.\\

Therefore, $b_{i}$ is the same as the cardinality of $B_i$ minus the number of all pairings in $(r(B_{i+1}),\overline{r}(B_i))$. Then the condition $b_i \geq a_{i+1}$ can be described in terms of $r(A_i)=\bfr^+(T^{(i)}),\overline{r}(B_i)=\overline{\bfr^-}(T^{(i)}),\overline{r}(A_{i+1})=\overline{\bfr^+}(T^{(i+1)}),r(B_{i+1})=\bfr^-(T^{(i+1)})$, and the condition is equivalent to the statement that the largest unpaired element in $D_i(T)$, after deleting all parings in $(C_i(T),D_i(T))$ is in $-\bfr^-(T^{(i+1)})$.\\

Now we describe the crystal operater $\te_i(T)$ for two cases: (1) $i_r >i$ (2) $i_r\leq i$. 

For case (1), it seems that we have to deal the case $i_r=i+1$ and the case $i_r>i+1$ seperately. If $i_r>i+1$, then we have $B'_i=B_i$ and $B'_{i+1}=B_{i+1}$, and 
$$V_{\beta_{i+1}}(M')=e_i(V_{\beta_{i+1}}(M))$$
which follows from Lemma \ref{vp} and Lemma \ref{tiv}.  
If $i_r=i+1$ we have $V_{\beta_{i+1}}(M')=V_{\beta_{i+1}}(M)$ and $B'_i=B_i$, but instead $B'_{i+1}$ has one less $i+1$ and one more $i$ compared to $B_{i+1}$.\\ 

However, we use a trick to consider two cases simultaneously. Consider a matrix $N=(n_{pq})$ such that $n_{pq}=0$ if $p<i$ or both $p>i+1$ and $q>i+1$, $n_{pq}=m_{pq}$ if $p>i+1$ and $q\leq i+1$. When $p=i$ or $i+1$, $n_{pq}$ is zero if $p<q$, $m_{pq}$ if $p>q$, and $m_{pq}/2$ if $p=q$. For example, if $i=2$ and $M$ is
$$\begin{bmatrix} 2&2&1&1 \\ 2&2&0&1 \\1&0&4&1 \\ 1&1&1&2\end{bmatrix},$$
then we define $N$ by
$$\begin{bmatrix} 0&0&0&0 \\ 2&1&0&0 \\1&0&2&0 \\ 1&1&1&0 \end{bmatrix}.$$
We also define $N'$ similarly for $M'$. Then by the construction, we have
$$e_i^P(N)=N'.$$ 
Indeed, $i$-th and $(i+1)$-th columns of $N$ (resp. $N'$) can be recovered from $A_i, A_{i+1}$ (resp. $A'_i$, $A'_{i+1}$). In the above example, we have
$$M'=\begin{bmatrix} 2&2&1&1 \\ 2&2&0&2 \\1&0&4&0 \\ 1&2&0&2 \end{bmatrix}\quad N'=\begin{bmatrix} 0&0&0&0 \\ 2&1&0&0 \\1&0&2&0 \\ 1&2&0&0 \end{bmatrix}$$
Note that $P$-tableau for $N$ is $V_{\beta_{i+1}}(M) \leftarrow B_{i+1} \leftarrow B_i$ and $P$-tableau for $N'$ is $V_{\beta_{i+1}}(M')\leftarrow B'_{i+1}\leftarrow B'_i$, hence we have 
$$V_{\beta_{i+1}}(M')\leftarrow B_{i+1}\leftarrow B_i=e_i(V_{\beta_{i+1}}(M) \leftarrow B'_{i+1}\leftarrow B'_i)$$
Recall that $r(A_i),\overline{r}(A_{i+1})$ are positions of $i$ and $i+1$ in $V_{\beta_{i+1}}(M) \leftarrow B_{i+1}\leftarrow B_i$, and $r(A'_i),\overline{r}(A'_{i+1})$ are positions of $i$ and $i+1$ in $V_{\beta_{i+1}}(M') \leftarrow B'_{i+1}\leftarrow B'_i$ hence the description is the same as the one described in Section 3.4. Moreover, since $Q$-tableau of $N$ and $N'$ are equal, we showed that $\overline{r}(B_i),r(B_{i+1})$ does not change under $\te_i$.\\

For case (2), first note that $V_{\beta_{i+1}}(M')=V_{\beta_{i+1}}(M)$, where $M'=\te_i(M)$. We also know that $B'_i$ has one more $i_r$ than $B_i$ and $B'_{i+1}$ has one less $i_r$ than $B_{i+1}$. We first show that $r(A_i),\overline{r}(A_{i+1})$ are unchanged. Note that they are row positions of $i$ and $i+1$ in the tableau $(V_{\beta_{i+1}}(M)\leftarrow B_{i+1})\leftarrow B_i$. For this case, we have
$$(V_{\beta_{i+1}}(M)\leftarrow B_{i+1})\leftarrow B_i = (V_{\beta_{i+1}}(M')\leftarrow B'_{i+1})\leftarrow B'_i$$
since $V_{\beta_{i+1}}(M)=V_{\beta_{i+1}}(M')$ and $P$-tableaux when inserting $B_{i+1}$ and $B_i$ does not change since $B_i,B_{i+1}$ and $B'_i,B'_{i+1}$ differ by the operator $e^Q$. Therefore, $r(A_i),\overline{r}(A_{i+1})$ are unchanged.\\

On the other hand, $\overline{r}(B_i)$ and $r(B_{i+1})$ records two horizontal strips in $Q$-tableaux when inserting $B_{i+1}$ and $B_i$. Since this $Q$-tableau is somewhat different from the definition of $Q(M)$ defined in Section 2, let us describe this $Q$-tableau precisely. We first define a $\overline{Q}$-tableau for any matrix $M$ by the recording tableau when row-inserting the inverse reading word of $M$.

\begin{example}
Let $m=4$ and let $N$ be 
$$\begin{bmatrix} 0&0&0&0 \\ 2&1&0&0 \\1&1&1&0 \\ 1&1&1&0 \end{bmatrix}.$$
Then the inverse reading word of $N$ is $(123)(123)(112)()$, where the parentheses imply the row number of each number. Then $P$-tableau $P(N)$ and $Q$-tableau $\overline{Q}(N)$ are 
$$P(N)=\begin{ytableau}1&1&1&1&2\\2&2&3\\3 \end{ytableau}, \quad \overline{Q}(N)=\begin{ytableau}1&1&1&2&3\\2&2&3\\3 \end{ytableau}.$$
Note that entries in $(m+1-i)$-th row of $N$ contributes to entries of $i$ in $\overline{Q}(N)$. 
\end{example} 
It is well-known that $P(N)$ is the same as the one defined in Section 2, but $\overline{Q}(N)$ is clearly different from $Q(N)$: entries of $i$ in $\overline{Q}(N)$ are determined when inserting numbers at $(m+1-i)$-th row of $N$. However, there is a very simple relation between these $Q$-tableaux.
\begin{lemma}\label{Qrotation} For a $m \times m$ matrix $N$, consider a matrix $\psi(N)$ obtained from $N$ by rotating $180^\circ$ degrees. Then 
$$\overline{Q}(N)=Q(\psi(N)).$$
\end{lemma}
\begin{proof} For a sequence of integers $x_1\cdots x_\ell$, define $\psi(x_1\cdots x_\ell)$ by $(m+1-x_1)\cdots (m+1-x_\ell)$. Then it is clear that when we apply $\psi$ to the inverse reading word of $N$, we obtain the reading word of $\psi(N)$. Recall that $Q(N)$ is the recording tableau when column-inserting the word $\psi(N)$. Then one can show the lemma by using the observation that a decreasing subsequence of a word $x$ corresponds to an increasing subsequence of the word of $\psi(x)$, and Greene's theorem stating that the shape of $Q(N)$ is determined by the cardinality of the disjoint family of decreasing/increasing subsequences of the reading word (See \cite[A1.1.1]{Sta}). 
\end{proof}

Going back to our situation, we define $N,N'$ by the same $N,N'$ defined in the case $(1)$. Then by the construction, we have
$$e_i^Q(N)=N',$$

Also, $P$-tableau for $N$ is $V_{\beta_{i+1}}(M) \leftarrow B_{i+1} \leftarrow B_i$ as before, and we already showed that $P(N)=P(N')$ proving that $r(A_i),\overline{r}(A_{i+1})$ are unchanged. Note that $r(B_{i+1}), \overline{r}(B_i)$ are the row positions of $m-i$ and $m+1-i$ of $\overline{Q}(N)$. By lemma \ref{Qrotation} and the fact that $\psi(e_i^Q(N)))=f_{m-i}^Q(\psi(N))$, we have 
$$\overline{Q}(N')=\overline{Q}(e_i^Q(N))=Q(\psi(e_i^Q(N)))=Q(f_{m-i}^Q(\psi(N))=f_{m-i}(\overline{Q}(N)).$$

Therefore, $r(B'_{i+1}),\overline{r}(B'_i)$ are simply obtained from $r(B_{i+1}), \overline{r}(B_i)$ by applying $f_{m-i}$ operator. Note that the element $x$ in $r(B_{i+1})$ that are moved to $\overline{r}(B'_{i})$ is the smallest unpaired number 
in $r(B_{i+1})$, after deleting pairs $(p',q')$ with $p'<q'$ in $(r(B_{i+1}),\overline{r}(B_i))$. This $x$ corresponds to the largest unpaired number $-x$ in $-r(B_{i+1})=-{\bfr^-}(T^{(i+1)})$ so $\te_i(T)$ is the same as the one described in Section 3.4. Hence we proved that the description of $\te_i(T)$ for both cases is the same as the one given in Section 3.4. \\

Now we show Theorem \ref{main2}: the set $\SSOT_g(\hat\mu,m)$ is isomorphic to $\mathcal{B}(\mu)$ as $\spm$-crystals.\\

{\it Proof of Theorem \ref{main2}.}
For a SSOT $T \in \SSOT_g(\hat\mu,m)$, consider a fixed skew tableaux $R$ such that $\inside(R)=\hat \mu$, $\outside(R)=\emptyset$, $c(R)\leq g$ and weight $\gamma$ with $\ell(\gamma)=\ell$. Then by joining two SSOT $T$ and $R$, one can make an element in $\SSOT_g(\emptyset,m+\ell)$, which we denote by $T \cup R$. For a non-negative integer $i \leq m-1$, consider $\te_i(T\cup R)$. Since $\te_i$ acts locally and $i <m$, we have $\te_i(T\cup R)= T' \cup R$ for some $T' \in \SSOT_g(\hat\mu,m)$, and $T'$ is independent of $R$. Now we simply define $\te_i(T)$ by $T'$. The construction shows that the set $\SSOT_g(\hat\mu,m)$ can be considered as a sub-$\spm$-crystal of $\SSOT_g(\emptyset,m+\ell)$, so $SSOT_g(\hat\mu,m)$ is a direct sum of $\spm$-crystals for irreducible representations. We are left to show that the set $\SSOT_g(\hat\mu,m)$ is isomorphic to $\mathcal{B}(\mu)$. To show this, we need to compute the highest weight vector. One can show that $\SSOT_g(\hat\mu,m)$ has the unique highest weight SSOT $T$, which has the row number
$$\bfr(T)=(11\cdots 11)(22\cdots 22) \cdots.$$
whose weight is $\hat\mu=(g-\mu_m,\cdots,g-\mu_1)$, so the crystal weight is $(\mu_m,\ldots,\mu_1)=\mu_m \epsilon_{\overline{m}}+\mu_{m-1}\epsilon_{\overline{m-1}}+\cdots+\mu_1 \epsilon_{1}$ which we identified by $\mu$.\qed\\

\section{applications}\label{app}

\subsection{Double crystal structure on $\wedge ((\mathbb{C}^{2m})^g)$}
In this section, we explain the fact that a SSOT appears naturally as a $Q$-tableau in symplectic version of RSK algorithm, providing the double crystal structure on 
$\wedge ((\mathbb{C}^{2m})^{g})$. First choose a basis of $i$-th component $\mathbb{C}^{2m}$ by $\varepsilon_{i,1},\varepsilon_{i,2}.\ldots,\varepsilon_{i,m},\varepsilon_{i,\overline{m}},\ldots,\varepsilon_{i,\overline{1}}.$ Then a basis of $\wedge ((\mathbb{C}^{2m})^{g})$ can be taken by
$$\{\wedge_{i=1}^g( \varepsilon_{i,j_{i,1}}\wedge \ldots, \wedge \varepsilon_{i,j_{i,\ell_i}})\}$$
and we identify the $i$-th component by a column tableau $$\begin{ytableau}j_{i,1}\\j_{i,2}\\ \cdots \\ j_{i,\ell_i}\end{ytableau}$$
where $j_{i,1}<\cdots<j_{i,\ell_i}$ with respect to the ordering $1<2<\cdots<m<\overline{m}<\cdots<\overline{1}$. We consider this tableau as a Kashiwara-Nakajima tableau (KN tableau in short). Therefore, we can identify each element in the basis as $g$-tuble of one-column KN tableaux. Then we use Lecouvey's insertion \cite{Lec1} to provide one $P$-tableau, which is a KN-tableau,and one $Q$-tableau, which is a (transpose of) SSOT. Note that the case when $\ell_i=1$ for all $i$ is already described in Lecouvey's work \cite{Lec1} and in this case $Q$-tableau is a SSOT with a weight $(1,1,\ldots,1)$. Hence, instead of providing exact definitions of KN tableaux as well as insertions, we provide an example.

\begin{example}
Assume that $g=3$, and we have a following $3$-tuple of one-column KN tableaux
$$v_3\otimes v_2 \otimes v_1=\begin{ytableau} \overline{3} \end{ytableau} \otimes \begin{ytableau} 1\\2\\ \overline{1} \end{ytableau} \otimes \begin{ytableau} 1\\2\\3\end{ytableau}.$$
To define $P,Q$-tableaux of this element (which we will denote by $P$ and $Q$), we define the reading word of $v_3\otimes v_2 \otimes v_1$ by the concatenation of reading words of each $v_i$. In our example, the reading word is $12312\overline{1}\overline{3}$. Then we insert this word into the empty KN tableau. Inserting the reading word of $v_1$ into the empty tableau gives 
$$P_1=\begin{ytableau} 1\\2\\3\end{ytableau}, Q_1=\left(\emptyset,\ydiagram{1},\ydiagram{1,1},\ydiagram{1,1,1}\right).$$
Inserting the reading word of $v_2$ into $P_1$ provides
$$P_2=\overline{1}\rightarrow (2\rightarrow (1 \rightarrow P_1))=\overline{1}\rightarrow \begin{ytableau} 1&1\\2&2\\3\end{ytableau}=\begin{ytableau} 1&2\\2\\3\end{ytableau}$$
where $x\rightarrow C$ is the insertion of a letter $x$ into a KN tableau $C$ described in \cite{Lec1} and
$$Q_2=\left(\ydiagram{1,1,1},\ydiagram{2,1,1}, \ydiagram{2,2,1},\ydiagram{2,1,1} \right).$$
Finally, inserting $v_3$ into $P_2$ provides
$$P=P_3= \begin{ytableau} 1&2 \\ 2 \end{ytableau}$$
and 
$$Q_3=\left( \ydiagram{2,1,1}, \ydiagram{2,1} \right)$$
To obtain a SSOT $Q$ from $Q_1,Q_2,Q_3$, we simply consider the concatenation of transposes of $Q_i$, providing the row sequence of $Q$ as
$$r(Q)=(111)(12\overline{1})(\overline{1})$$
\end{example}
A proof to show that the tableau $Q$ is indeed a SSOT is very similar to the standard case, where weight is $(1,1,\ldots,1)$, whose proof is given in \cite[Proposition 5.1.3]{Lec1}, hence we omit the proof. Now it is clear that we have $\spm\oplus\mathfrak{sp}_g$-crystal structure on $\wedge ((\mathbb{C}^{2m})^g)$, where crystal operators for $\spm$ acts only on $P$, and crystal operators for $\mathfrak{sp}_g$ acts only on $Q$.

\subsection{Symplectic Schur times Schur}
For a partition $\lambda$ of length $\leq m$, let $\chi_\lambda$ be the irreducible character of $Sp(2m)$ indexed by $\lambda$, and let $s_\lambda:=s_\lambda(x_1^{\pm},\ldots,x_m^{\pm})$ be the Schur polynomial in $x_1^{\pm},\ldots,x_m^{\pm}$ indexed by $\lambda$. Note that $s_\lambda$ is the character of the induced representation $L(\lambda)\downarrow^{GL(2m)}_{Sp(2m)}$ where $L(\lambda)$ is the irreducible representation of $GL(2m)$ indexed by $\lambda$. In this section, we present the conjectured formula computing the number of $\chi_\nu$ in $\chi_\lambda s_\mu$, generalizing the dual Pieri rule(Theorem \ref{dualpieri1}).

\begin{conj} \label{chis} For partitions $\lambda,\mu,\nu$ of length $\leq m$, the number of $\chi_{\nu}$ in $\chi_{\lambda} s_{\mu}$ is the same as the number of skew SSOT $T$ such that $\inside(T)=\lambda'$, $\outside(T)=\nu'$, $\wt(T)=\mu'$, $c(T)\leq m$ and $\varepsilon_i(T)=0$ for $i=1,2,\ldots,m-1$.
\end{conj}

Note that we have $\mathfrak{gl}_m$-crystal structure on the set of skew SSOT $T$ such that $\inside(T)=\lambda$, $\outside(T)=\nu$, $c(T)\leq m$ and the length of $\wt(T)$ is $(\alpha_1,\ldots,\alpha_m)$ for $0\leq \alpha_i \leq 2m$, since we provided $\spm$-crystal structure on the same set when $\lambda$ is empty and the operators $\te_i,\tf_i$ for $i>0$ are defined locally. Indeed, the same operators $\te_i,\tf_i$ for $i>0$ does satisfy the same Stembridge axioms (See \cite{Ste,BS}) which guarantees to poccess the $\mathfrak{gl}_m$-crystal structure. Therefore, one can define $\varepsilon_i(T)$ so that Conjecture \ref{chis} makes sense. \\

Also note that the case when $\mu$ is either $(k)$ or $(1^k)$ is known. When $\mu=(k)$, then $s_{(1^k)}$ is equal to $e_k:=e_k(x_1^\pm,\ldots,x_m^\pm)$ which is the character of $\wedge^k(\mathbb{C}^{2m})$ so the conjecture reduces to the dual Pieri rule. When $\mu=(1^k)$, the conjecture states that the number of $\chi_{\nu'}$ in the product $\chi_{\lambda'} h_k(x_1^{\pm},\ldots,x_m^{\pm})$ is the number of skew SSOT $T$ such that $\inside(T)=\lambda$, $\outside(T)=\nu$, $\wt(T)=(1^k,0^{m-k})$, $c(T)\leq m$, and $\varepsilon_i(T)=0$ for $i>0$. Therefore $\bfr(T)$ is of the form 
$$(i_1)(i_2)\cdots(i_k) () \cdots ()$$
where $i_1<i_2<\cdots<i_k$, and the number of such $T$'s is equal to
$$| \{ \gamma \mid \gamma \subset \lambda, \gamma \subset \nu, \lambda/\gamma, \nu/\gamma \textrm{ are vertical strips, and } |\lambda/\gamma|+|\nu/\gamma|=k\}.|$$
This is the Pieri rule for the symplectic Schur function proved by Sundaram \cite{Sun}.\\

For the rest of this subsection, we prove the conjecture when $\mu_1\leq 3$. When $\mu_1=2$, say $\mu'=(a_1,a_2)$, we know that $s_{\mu}=e_{a_1}e_{a_2} - e_{a_1+1}e_{a_2-1}$ by Jacobi-Trudi formula for Schur functions. Since we know that the number of $\chi_\nu$ in $\chi_\lambda e_{a_1}e_{a_2}$ is the same as the cardinality of the set
$$A_{a_1,a_2}:=\{ \textrm{ skew SSOT } T \mid \inside(T)=\lambda, \outside(T)=\nu, \wt=(a_1,a_2), c(T)\leq m\},$$
it is enough to provide an injection from $A_{a_1+1,a_2-1}$ to $A_{a_1,a_2}$ and the complement of the image is simply the set of $T$ such that $\varepsilon_1(T)=0$. The injection is simply the operator $\tf_1$ and we are done.\\

When $\mu_1=3$, say $\mu'=(a_1,a_2,a_3)$, we will use the same strategy: use Jacobi-trudi formula for $s_\mu$ and an induction on $\mu_1$. Now we also need local properties of $\te_i$ and $\varepsilon_i$ for $i=1,2$ derived from Stembride axioms. We list a few properties derived from Stembridge axioms below:

\begin{proposition}[Proposition 4.5, \cite{BS}] \label{local}
If $\langle \alpha_i,\alpha_j^\vee \rangle=-1$ and $\varepsilon_i(x)>0$, then exactly one of the following two possibilities is true:
\begin{enumerate}
\item $\varepsilon_j(\te_i(x))=\varepsilon_j(x), \quad \varphi_j(\te_i(x))= \varphi_j(x)-1,$
\item $\varepsilon_j(\te_i(x))=\varepsilon_j(x)+1, \quad \varphi_j(\te_i(x))= \varphi_j(x).$
\end{enumerate}
\end{proposition}

\begin{proposition}[Axiom S2, \cite{BS}] \label{localaxiom}
If $\varepsilon_i(x)>0$ and $\varepsilon_j(\te_i(x))=\varepsilon_j(x)+1>0$, then $\te_i\te_j(x)=\te_j\te_i(x)$ and $\varphi_i(\te_j(x))=\varphi_i(x)$.
\end{proposition}

\begin{cor} \label{localaxiom1}
Assume that $x$ satisfies assumptions of Proposition \ref{localaxiom} and $\langle \alpha_i,\alpha_j^\vee \rangle=-1$. Then $\varepsilon_i(\te_j(x))=\varepsilon_i(x)+1$.
\end{cor}

\begin{proposition}[Proposition 4.7, \cite{BS}] \label{local1}
If $\langle \alpha_i,\alpha_j^\vee \rangle=-1$, $\varepsilon_j(x)>0$, and $\varepsilon_i(\te_j(x))=\varepsilon_i(x)+1$, then $\varepsilon_j(\te_i\te_j(x))=\varepsilon_j(x)-1$.
\end{proposition}

Note that we also have dual version of above propositions when $\te_i,\varepsilon_i$ are properly replaced by $\tf_i,\varphi_i$, and we will also use the dual version.\\

Now we are ready to prove Conjecture \ref{chis} for $\mu_1=3$, say $\mu'=(a_1,a_2,a_3)$. First note that we have
\begin{align*} s_\mu&=\det \begin{bmatrix} e_{a_1} &e_{a_1+1} &e_{a_1+2} \\ e_{a_2-1} &e_{a_2} &e_{a_2+1} \\ e_{a_3-2} & e_{a_3-1} & e_{a_3} \end{bmatrix}\\
&= s_{(a_1,a_2)'}e_{a_3} - (s_{(a_1,a_2+1)'}+s_{(a_1+1,a_2)'} ) e_{a_3-1} +  s_{(a_1+1,a_2+1)'}  e_{a_3-2}
\end{align*}

Define the set $B_{(a_1,a_2,a_3)}$ by 
$$ \{ T \mid \inside(T)=\lambda, \outside(T)=\nu, \wt=(a_1,a_2,a_3), c(T)\leq m, \varepsilon_1(T)=0 \},$$
Then by an induction, it is enough to prove that the cardinality of the set
$$ B:= \{ T \mid \inside(T)=\lambda, \outside(T)=\nu, \wt=(a_1,a_2,a_3), c(T)\leq m, \varepsilon_1(T)=\varepsilon_2(T)=0 \}$$
is the same as $|B_{(a_1,a_2,a_3)}|-|B_{(a_1,a_2+1,a_3-1)}| - |B_{(a_1+1,a_2,a_3-1)}|+ |B_{(a_1+1,a_2+1,a_3-2)}|$. Note that $B_{(a_1,a_2,a_3)}$ contains the set $B$ with the complement consisting of $T \in B_{(a_1,a_2,a_3)}$ with $\varepsilon_2(T)\geq 1$. Let us call this complement by $C$.\\

We will define an injection $h$ from $C$ to the disjoint union $B_{(a_1,a_2+1,a_3-1)} \cup B_{(a_1+1,a_2,a_3-1)}$. Let $T$ be an element in $C$. Then $\te_i(T)$
lies in the set
$$\{T' \mid \inside(T')=\lambda, \outside(T')=\nu, \wt(T')=(a_1,a_2+1,a_3-1), c(T')\leq m, \varepsilon_1(T')\leq 1 \},$$

If $\varepsilon_1(T')=0$, then we set $h(T)=\te_2(T)=T' \in B_{(a_1,a_2+1,a_3-1)}$. By proposition \ref{local}, we know that $\varepsilon_1(T)\leq \varepsilon_1(T')=0$ hence the inverse map $g^{-1}: B_{(a_1,a_2+1,a_3-1)} \rightarrow C$ is well-defined.\\

If $\varepsilon_1(T')=1$, then we define $h(T)$ by $\te_1 \te_2(T)=\te_1(T')$. Let $R=h(T)$. Then the image of this $h$ lies in the set
$$\{ R \in B_{(a_1+1,a_2,a_3-1)} \mid \varepsilon_1(\tf_2\tf_1(R))=0 \}.$$
 Note that $\varepsilon_1(R)=0$ implies $\varepsilon_1(\tf_1(R))=1$ hence $\varepsilon_1(\tf_2 \tf_1(R)) \leq 1$ by Proposition \ref{local}. Therefore, the complement of the image of $h$ is the set 
$$D:=\{ R \in B_{(a_1+1,a_2,a_3-1)}\mid \varepsilon_1(\tf_2\tf_1(R))=1 \}$$
Recall that $R$ satisfies $\varepsilon_1(R)=0$ and $\varepsilon_1(\tf_1(R))=1$.
 
Finally, we give a bijection $h'$ between the complement $D$ and $B_{(a_1+1,a_2+1,a_3-1)}$, which is simply $\te_2$. To define the map, we need to show that $\varepsilon_2(R)\geq 1$ from given properties of $R$. By Proposition \ref{localaxiom} and Corollary \ref{localaxiom1} with $i=2,j=1, x=T=\tf_2\tf_1(R)$, the condition $\varepsilon_1(T)=\varepsilon_1(\te_2(T))=1$ implies that $\te_1\te_2 (T)=\te_2\te_1(T)$ and $\varepsilon_2(\te_1(T))=\varepsilon_2(T)+1$. Therefore, we have
$$\varepsilon_2(R)=\varepsilon_2(\te_1\te_2(T))=\varepsilon_2(\te_2\te_1(T))=\varepsilon_2(\te_1(T))-1=\varepsilon_2(T)>0.$$
For the last equality, recall that $\varepsilon_2(T)>0$ since $T'=\te_2(T)$ is defined.\\

Now we need to show that $h'$ has an inverse which is simply $\tf_2: B_{(a_1+1,a_2+1,a_3-2)}\rightarrow B_{(a_1+1,a_2,a_3-1)}$ and show that the image is exactly $D$. Note that for $R'\in B_{(a_1+1,a_2+1,a_3-2)}$, $\tf_2(R)$ and $\tf_2\tf_1\tf_2(R')$ are always defined since $a_1\geq a_2\geq a_3$. To show that the image is exactly $D$, we need to prove that $\varepsilon_1(\tf_2(R'))=0$ and $\varepsilon_1(\tf_2\tf_1\tf_2(R'))=1$ from the condition $\varepsilon_1(R')=0$. Let $R,T',T$ denote $\tf_2(R'),\tf_1\tf_2(R'),\tf_2\tf_1\tf_2(R')$ as before.\\

By Proposition \ref{local}, one can show that $\varepsilon_1(\tf_2(R'))\leq \varepsilon_1(R')$ always holds hence $\varepsilon_1(\tf_2(R'))=0$ because $\varepsilon_1(R')=0$. Therefore, it is enough to show that $\varepsilon_1(\tf_2\tf_1\tf_2(R'))=1$. We will first prove that $\varphi_2(T')=\varphi_2(R)$. Since we have $\varepsilon_1(R')=\varepsilon_1(R)=0$ we know that $\varphi_1(R')=\varphi_1(R)-1$ by Proposition \ref{local}. By the dual version of Proposition \ref{local1} for $x=R'$, $\varphi_2(\tf_1\tf_2(R'))=\varphi_2(R')-1$ which is the same as $\varphi_2(\tf_2(R'))$. Therefore, $\varphi_2(R)=\varphi_2(T')=\varphi_2(\tf_1(R))$. Now we can apply the dual version of Proposition \ref{localaxiom} for $x=R$ and obtains
$$\tf_1\tf_2(R)=\tf_2\tf_1(R) \quad \textrm{and} \quad \varepsilon_1(\tf_2(R))=\varepsilon_1(R).$$
We know that $\varepsilon_1(R)=0$ therefore

$$\varepsilon_1(\tf_2\tf_1\tf_2(R'))=\varepsilon_1(\tf_2\tf_1(R))=\varepsilon_1(\tf_1\tf_2(R))=\varepsilon_1(\tf_2(R))+1=1$$
and Conjecture \ref{chis} for $\mu_1=3$ is proved.

\section*{acknowledgment}
The author thanks Jae-Hoon Kwon, Kyu-Hwan Lee, and Sejin Oh for helpful discussions. In particular, the author thanks Jae-Hoon Kwon for Section \ref{app}.1. This work was supported by Research Resettlement Fund for the new faculty of Seoul National University.

\end{document}